\documentclass[oneside,11pt]{amsart}

\usepackage{mathtools}
\usepackage[backend=bibtex,style=alphabetic,date=year]{biblatex}
\bibliography{biblio}
\usepackage{amssymb,amsmath,amsthm}
\usepackage{calligra,mathrsfs}
\usepackage[foot]{amsaddr}
\usepackage{tikz-cd}
\usepackage{hyperref}
\usepackage{comment}

\usepackage[textwidth=17cm,top=2cm,bottom=2cm]{geometry}

\newtheorem{thm}{Theorem}[section]
\newtheorem{prop}[thm]{Proposition}
\newtheorem{lem}[thm]{Lemma}
\newtheorem{cor}[thm]{Corollary}

\theoremstyle{definition}
\newtheorem{defn}[thm]{Definition}

\theoremstyle{remark}

\newtheorem{rmk}[thm]{Remark}
\newtheorem{set}[thm]{Settings}
\newtheorem{que}[thm]{Question}

\numberwithin{equation}{section}

\newenvironment{claim}[1]{\par\noindent\underline{Claim:}\space#1}{\par}
\newenvironment{claimproof}[1]{\par\noindent\underline{Proof of claim:}\space#1}{\hfill $\blacksquare$\par}

\newcommand{\bbC}{\mathbb{C}}

\newcommand{\bbF}{\mathbb{F}}

\newcommand{\bbN}{\mathbb{N}}

\newcommand{\bbP}{\mathbb{P}}
\newcommand{\bbQ}{\mathbb{Q}}

\newcommand{\bbZ}{\mathbb{Z}}

\newcommand{\calC}{\mathcal{C}}
\newcommand{\calD}{\mathcal{D}}
\newcommand{\calE}{\mathcal{E}}
\newcommand{\calF}{\mathcal{F}}
\newcommand{\calG}{\mathcal{G}}

\newcommand{\calL}{\mathcal{L}}

\newcommand{\calN}{\mathcal{N}}
\newcommand{\calO}{\mathcal{O}}
\newcommand{\calP}{\mathcal{P}}

\newcommand{\frakT}{\mathfrak{T}}

\DeclareMathOperator{\Hom}{Hom}
\DeclareMathOperator{\SheafHom}{\mathscr{H}\text{\kern -5pt {\calligra\large om}}\,}
\DeclareMathOperator{\End}{End}
\DeclareMathOperator{\Aut}{Aut}
\DeclareMathOperator{\Id}{Id}

\DeclareMathOperator{\Spec}{Spec}

\DeclareMathOperator{\Sym}{Sym}

\DeclareMathOperator{\GL}{GL}
\DeclareMathOperator{\pr}{pr}

\DeclareMathOperator{\Supp}{Supp}

\DeclareMathOperator{\Alb}{Alb}
\DeclareMathOperator{\rank}{rk}

\begin{document}


\title[Associated bundles and Beauville--Bogomolov decomposition]{A criterion for associated bundles and the weak Beauville--Bogomolov decomposition}


\author{Yang Zhang}
\address{Institut de Math\'ematiques, \'Ecole polytechnique f\'ed\'erale de Lausanne}
\email{yang.zhang@epfl.ch}





\begin{abstract}
We give a criterion for a fibration to be an associated fibre bundle. 
As an application, we show that, under suitable singularity and Cartier-index assumptions, any surjective fibration $X \rightarrow Y$ with nef $-K_{X/Y}$ is an associated fibre bundle. 
From this, we also deduce a weak Beauville--Bogomolov decomposition for varieties with nef anticanonical divisor.
\end{abstract}


\maketitle


\setcounter{tocdepth}{1}
\tableofcontents


\section{Introduction}
An isotrivial family is a family of schemes, or more generally, of pairs $(X,\Delta)\rightarrow Y$ such that every geometric fibre is isomorphic. 
They usually occur as obstructions to the existence of a fine moduli space, and are therefore heavily studied in the area. 
A way to construct an isotrivial family is to take the fibre bundle associated to a torsor.

\begin{defn}
	Over a field $k$, let $G$ be a group scheme with an action on a pair $(W,\Theta)$, and let $Y' \rightarrow Y$ be a $G$-torsor. 
	Let $(X,\Delta):= (Y'\times W, Y'\times \Theta) /G$ be the diagonal quotient. 
	We call the natural map $(X,\Delta) \rightarrow Y$ the \textbf{$(W,\Theta)$-bundle associated to $Y' \rightarrow Y$}. 
\end{defn}

Here, a morphism from a pair to a variety $(X,\Delta) \rightarrow Y$ is nothing but just a map $X\rightarrow Y$. 
And one can easily check that all geometric fibres are isomorphic to $(W,\Theta)_{\overline k}$. 

A classical example for associated bundles arising naturally from the geometry is the Beauville--Bogomolov decomposition (see e.g. \cite{bogomolovDecompositionKahlerManifolds1974, beauvilleVarietesKahleriennesDont1983}), which states that for a complex projective manifold $X$ with $c_1(X)=0$, the Albanese $X \rightarrow \Alb_X$ is a fibre bundle associated to $\widetilde{\Alb_X} \rightarrow \Alb_X$, where $\widetilde{\Alb_X}$ is the topological universal cover of $\Alb_X$, considered (analytically) as a $\pi_1(\Alb_X,0)$-torsor. 

Over the recent decades, a substantial amount of work has been done (\cite{kawamataCharacterizationAbelianVarieties1981, ambroModuliBDivisor2005, demaillyCompactKahlerManifolds, caoAlbaneseMapsProjective2018, xuHomogeneousFibrationsLog2020}) to weaken the assumptions in the decomposition. 
It was noticed that we can allow boundaries and mild singularities, and the negativity of the first Chern class suffices to constrain the geometry of the Albanese. 
Recently Matsumura and Wang proved: 

\begin{thm}[{\cite[Corollary 4.2]{matsumuraStructureTheoremProjective2025}}]\label{thm: associated bundle complex case}
	Let $(X,\Delta)$ be a normal klt projective pair over $\bbC$ such that $-K_X-\Delta$ is nef, then the Albanese $(X,\Delta)\rightarrow A$ is a fibre bundle associated to $\widetilde A \rightarrow A$. 
\end{thm}

In positive characteristics, a parallel sequence of studies on analog statements
(\cite{ejiriWhenAlbaneseMorphism2019, patakfalviWeakBeauvilleBogomolovDecomposition2025})
has been also carried out and remains an active area. 
A recent result is:

\begin{thm}[{\cite[Theorem 5.7]{ejiriSplittingAlgebraicFiber2023a}}]\label{thm: isotriviality positive char}
	Let $(X,\Delta)$ be a globally $F$-split strongly $F$-regular projective pair over a perfect field $k$ of characteristic $p>0$, such that $-K_X-\Delta$ is nef and of Cartier index coprime to $p$. 
	Then the Albanese $X\rightarrow A$ is isotrivial.  
\end{thm} 

This result says however nothing on the boundary. 
Ejiri posted then in the same paper the following question, asking if one can prove a more general isotriviality result over an arbitrary base. The parallel question over complex numbers is true by \cite[Theorem 4.1]{matsumuraStructureTheoremProjective2025}

\begin{que}
	Let $f:(X,\Delta)\rightarrow Y$ be a morphism between projective varieties over a perfect field $k$ of characteristic $p>0$, such that the geometric generic fibre $(X_{\overline \eta}, \Delta_{\overline \eta})$ is strongly $F$-regular and $-K_{X/Y}-\Delta$ is nef. 
	Is it true that $f$ is isotrivial?
\end{que}

This paper is aimed to give a proof to the question, and discuss some consequences and related open problems. 

\subsection{Associated bundle criterion and the weak Beauville--Bogomolov decomposition}
In Section \ref{section: Criterion of being associated fibre bundles}, we give an affirmative answer to the question. 
Indeed, we show a stronger result that under mild extra assumptions, the morphisms described above are even associated fibre bundles. 
To this end, we apply the theory of S-fundamental group schemes developed by Langer. 
In \cite{langerSfundamentalGroupScheme2011} and \cite{langerSfundamentalGroupScheme2012}, Langer introduced for a smooth projective variety $Y$ with a rational point $y\in Y(k)$ a pro-algebraic group scheme $\pi_1^S(Y,y)$, called the S-fundamental group scheme, and a principal $\pi_1^S(Y,y)$-bundle $\widetilde Y^S$, called the S-universal cover. 
We shall consider them as an algebraic replacement for the topological fundamental group and the universal cover. 
As a result, we prove the following criterion for being an associated bundle. 

\begin{thm}[cf. Theorem \ref{thm: main theorem in text}]\label{thm: main theorem}
	Let $(X,\Delta)$ be a projective pair and $Y$ be a smooth projective variety over a perfect field $k$ of characteristic $p\geq 0$. 
	Let $f:(X,\Delta)\rightarrow Y$ be a morphism, and let $y\in Y(k)$ be a $k$-point. 
	We denote with $(X_y,\Delta_y)$ the fibre of $f$ over $y$. 
	Assume there exists an $f$-ample line bundle $\calL$, such that
	\begin{enumerate}
		\item the sheaves $f_*(\calL^{\otimes m})$ are numerically flat for all $m>0$,
	\end{enumerate} 
	and for every irreducible component $D$ of $\Delta$, there is a positive integer $r$ such that
	\begin{enumerate}
		\item[(2)] $f_*(\calL^{\otimes m}(-rD))$ are numerically flat for all $m>0$. 
	\end{enumerate}
	Then we can find a linear algebraic group $G$ which is a quotient of $\pi_1^S(Y,y)$, a principal $G$-bundle $P\rightarrow Y$, such that $f$ is an $(X_y,\Delta_y)$-bundle associated to $P \rightarrow Y$. In particular, $f$ is an $(X_y,\Delta_y)$-bundle associated to $\widetilde Y^S \rightarrow Y$.
\end{thm}

Let us give an outline of the proof. The proofs of Theorem \ref{thm: associated bundle complex case}, Theorem \ref{thm: isotriviality positive char} both construct a relatively ample line bundle such that its pushforward to $A$ is numerically flat. 
This idea can be traced back to \cite{demaillyCOMPACTCOMPLEXMANIFOLDS}. 
Briefly speaking, over the complex numbers, one can show that given a numerically flat vector bundle $\calE$ on a variety $X$, the pullback $\widetilde \calE$ to the universal cover $\widetilde X$ is trivialized and admits a flat connection. 
Any morphism between pullbacks of numerically flat bundles must respect this connection and can be shown constant. 
Since the datum of $X$ is completely encoded in the multiplication map $\Sym^m f_*\calL \rightarrow f_*\calL^m$, which by the argument above must be constant when pulling back to the universal cover, the family $X\rightarrow Y$ is indeed isotrivial.
This idea however relies heavily on analytic tools. 
Ejiri's approach in \cite{ejiriSplittingAlgebraicFiber2023a} relies on very concrete descriptions of numerically flat bundles on abelian varieties and could not be generalized to a general base, nor does it say anything about the boundary.  
To prove Theorem \ref{thm: main theorem}, we use Langer's Tannakian approach and consider the behaviour of pullback of numerically flat bundles on the $S$-universal cover. 
This framework uses only purely algebraic techniques, hence Theorem \ref{thm: main theorem} works free from characteristic constraints. 
In particular, this allows us to prove the general associated bundle statement in positive characteristics, and also include the boundary. 

\begin{thm}[cf. Theorem \ref{thm: associated bundle relative anti-nef positive characteristic in text}]\label{thm: associated bundle relative anti-nef positive characteristic}
	Let $f:(X,\Delta) \rightarrow Y$ be a surjective fibration over a perfect field $k$ of characteristic $p>0$, such that
	\begin{enumerate}
		\item $(X,\Delta)$ is a Cohen-Macaulay projective pair, 
		\item $Y$ is a smooth projective variety with a $k$-rational point $y\in Y(k)$, 
		\item the geometric generic fibre $(X_{\overline \eta}, \Delta_{\overline \eta})$ is strongly $F$-regular,
		\item the relative log-anti-canonical divisor $-K_{X/Y}-\Delta$ is nef and of Cartier index coprime to $p$.
	\end{enumerate}
	Then $f$ is an $(X_y,\Delta_y)$-bundle associated to the S-universal cover $\widetilde Y^S \rightarrow Y$.  
\end{thm}

When the base is the Albanese, we have $K_{X/A}=K_X$, and we are able to give a weak version of the Beauville--Bogomolov decomposition in the case where $-K_X$ is nef:

\begin{thm}[cf. Corollary \ref{cor: associated bundle abelian variety positive characteristic in text}]\label{thm: associated bundle abelian variety positive characteristic}
	Let $(X,\Delta)$ be a globally $F$-split strongly $F$-regular projective pair over a perfect field $k$ of characteristic $p>0$, such that $-K_X-\Delta$ is nef and of Cartier index coprime to $p$. 
	Let $(X,\Delta)\rightarrow A$ be the Albanese of $X$, and assume that $A$ has a rational point $0\in A(k)$. 
	Then $(X,\Delta)\rightarrow A$ is an $(X_0,\Delta_0)$-bundle associated to the S-universal cover $\widetilde A^S \rightarrow A$. 
\end{thm}

\subsection{A finiteness result when $-K_X-\Delta$ is semi-ample or base field is finite}
In Section \ref{section: Bertini}, we prove that when the base field $k$ is infinite and $-K_X-\Delta$ is semi-ample, we can trivialize the associated bundle over the Albanese after a finite pullback. 

\begin{thm}[cf. Corollary \ref{cor: associated bundle semi-ample in text}]\label{thm: associated bundle semi-ample}
	Let $(X,\Delta)$ be a globally $F$-split strongly $F$-regular projective pair over an infinite perfect field $k$ of characteristic $p>0$, such that $m(-K_X-\Delta)$ is Cartier and base point free for some integer $m$ coprime to $p$. 
	Let $(X,\Delta)\rightarrow A$ be the Albanese of $X$, and assume that $A$ has a rational point $0\in A(k)$. 
	Then $(X,\Delta)\rightarrow A$ is an $(X_0,\Delta_0)$-bundle associated to an isogeny $A'\rightarrow A$. 
\end{thm}

Theorem \ref{thm: associated bundle semi-ample} is known to hold true in \cite[Theorem 5.8]{ejiriSplittingAlgebraicFiber2023a}. 
We use here a different approach. 
Our proof starts from the Beauville--Bogomolov decomposition for log-$K$-trivial pairs given by \cite[Theorem 1.13]{patakfalviWeakBeauvilleBogomolovDecomposition2025}, and the following Bertini-type theorem then allows us to find a $K_X$-complement boundary without damaging the $F$-singularity properties. 
We state the theorem here as it is of independent interest. 

\begin{thm}[cf. Theorem \ref{thm: Bertini of GFS+SFR in text}]\label{thm: Bertini of GFS+SFR}
	Let $(X,\Delta)$ be a strongly $F$-regular proper pair over an $F$-finite field $k$ of characteristic $p>0$, such that
	\begin{enumerate}
		\item $H^0(X,\calO_X)$ is separable over $k$, and
		\item $K_X+\Delta$ has Cartier index coprime to $p$.
	\end{enumerate}
	Let $\Delta'$ be an effective $\bbQ$-Cartier $\bbQ$-divisor such that $(X,\Delta+\Delta')$ is globally $F$-split, and $m>1$ an integer coprime to $p$ such that $m\Delta'$ is Cartier and $|m\Delta'|$ is base point free. 
	Then for a general member $\Gamma\in |m\Delta'|$, the pair $(X,\Delta+\frac{1}{m}\Gamma)$ is globally $F$-split and strongly $F$-regular. 
\end{thm}

When the base field $k$ is finite, then the finite decomposition holds readily when $-K_X-\Delta$ is nef. 
Note also that the existence of a rational point $0\in A(k)$ is guaranteed by \cite{langAlgebraicGroupsFinite1956}.

\begin{thm}[cf. Theorem \ref{thm: associated bundle finite field in text}]\label{thm: associated bundle finite field}
	Let $(X,\Delta)$ be a globally $F$-split strongly $F$-regular projective pair over a finite field $k$ of characteristic $p>0$, such that $-K_X-\Delta$ is nef and of Cartier index coprime to $p$. 
	Then the Albanese $(X,\Delta)\rightarrow A$ is an $(X_0,\Delta_0)$-bundle associated to an isogeny $A'\rightarrow A$, where $0\in A(k)$ is a rational point. 
\end{thm}

\subsection{Consequences on base field $\bbC$}
As we mentioned, the statement in Theorem \ref{thm: main theorem} is characteristic free. 
So a statement parallel to Theorem \ref{thm: associated bundle relative anti-nef positive characteristic} holds true over the complex numbers. 

\begin{thm}[cf. Theorem \ref{thm: associated bundle relative anti-nef complex numbers in text}]\label{thm: associated bundle relative anti-nef complex numbers}
	Let $f:(X,\Delta) \rightarrow Y$ be a surjective fibration such that
	\begin{enumerate}
		\item $(X,\Delta)$ is projective klt,
		\item $Y$ is smooth projective with a point $y\in Y$,
		\item $-K_{X/Y}-\Delta$ is nef.
	\end{enumerate}
	Then $f$ is an $(X_y,\Delta_y)$-bundle associated to $\widetilde Y^S\rightarrow Y$.  
\end{thm}
Similarly we get also the following parallel to Theorem \ref{thm: associated bundle abelian variety positive characteristic}:

\begin{thm}[cf. Corollary \ref{cor: associated bundle abelian variety complex numbers in text}]\label{thm: associated bundle abelian variety complex numbers}
	Let $(X,\Delta)$ be a klt projective pair over $\bbC$, such that $-K_X-\Delta$ is nef. 
	Then the Albanese $(X,\Delta)\rightarrow A$ is an $(X_0,\Delta_0)$-bundle associated to $\widetilde A^S \rightarrow A$. 
\end{thm}

Over complex numbers there is a natural map $\widetilde A \rightarrow \widetilde A^S$ from the topological universal cover induced by a universal property, and $\widetilde A^S$ is more rigid than $\widetilde A$ in the sense that $H^0(\widetilde A^S,\calO_{\widetilde A^S})=k$ (see Proposition \ref{prop: rigidity of universal cover}). 
Hence we can directly recover Theorem \ref{thm: associated bundle complex case} from Theorem \ref{thm: associated bundle abelian variety complex numbers} without arguing on flat connections with analytic methods. 

There is of course a parallel statement to Theorem \ref{thm: Bertini of GFS+SFR} over the complex numbers by replacing ``strongly $F$-regular globally $F$-split" with ``klt", whose proof is just a simple discrepancy computation, and which directly implies the parallel statement of Theorem \ref{thm: associated bundle semi-ample} over the complex numbers. 

\subsection{Homogeneous fibration}
Finally in Section \ref{section: homogeneous fibration}, we see an application using the structure of associated bundles. 
We stress that the geometry of an associated fibre bundle is more constrained than an isotrivial map. 
Recall the definition of a homogeneous morphism on an abelian variety. 

\begin{defn}\label{defn: homogeneous fibration}
	A morphism $f:(X,\Delta)\rightarrow A$ to an abelian variety is called \textbf{homogeneous} if $f$ is isomorphic to its pullback along any translation on $A$, or equivalently, for any $a\in A(k)$, there exists an automorphism $\sigma_a:(X,\Delta)\rightarrow (X,\Delta)$ such that $f \circ \sigma_a=t_a \circ f$.
\end{defn}

We show that any fibre bundle associated to the S-universal cover on an abelian variety is essentially a homogeneous morphism.

\begin{thm}[cf. Theorem \ref{thm: associated bundles are homogeneous in text}]\label{thm: associated bundles are homogeneous}
	Let $f:(X,\Delta) \rightarrow A$ be a fibre bundle associated to $\widetilde A^S \rightarrow A$, where $A$ is an abelian variety over an algebraically closed field. 
	Then $f$ is homogeneous. 
\end{thm}

Combining with Theorem \ref{thm: associated bundle abelian variety positive characteristic} and Theorem \ref{thm: associated bundle abelian variety complex numbers} , we obtain directly

\begin{cor}[cf. Corollary \ref{cor: albanese is homogeneous in text}]\label{cor: albanese is homogeneous}
	Let $(X,\Delta)$ be either
	\begin{enumerate}
		\item a normal klt projective pair over $\bbC$, such that $-K_X-\Delta$ is nef, or
		\item a globally $F$-split strongly $F$-regular projective pair over an algebraically closed field $k$ of characteristic $p>0$, such that $-K_X-\Delta$ is nef and of Cartier index coprime to $p$.
	\end{enumerate}
	Then the Albanese $(X,\Delta)\rightarrow A$ is homogeneous. 
\end{cor}

In contrast, we remark that an isotrivial map to an abelian variety generally fails to be homogeneous.

\subsection*{Notations and conventions}
We work always over a field $k$, and the exact settings of $k$ will be specified in each section. 
A pair $(X,\Delta)$ is the datum of a normal variety and an effective $\bbQ$-divisor $\Delta$ on it. 
We assume also $K_X$ and all irreducible components of $\Delta$ are $\bbQ$-Cartier. 
By a fibration $f:(X,\Delta)\rightarrow Y$, we mean a morphism $X\rightarrow Y$ such that the natural map $\calO_Y \rightarrow f_*\calO_X$ is an isomorphism. 
We mean by an algebraic group a subgroup scheme of $\GL_n$. 
In particular, an algebraic group need not be reduced or connected. 
As we work also over non-algebraically closed fields, the Albanese $X\rightarrow A$ is referred to the initial object in the category of morphisms from $X$ to torsors over abelian varieties. 
In particular, $A$ might have no $k$-rational point.
See \cite{wittenbergAlbaneseTorsorsElementary2008a} for a modern treatment. 

\subsection*{Acknowledgements}
The topics studied in this paper were suggested by my PhD supervisor Zsolt Patakfalvi. 
I would like to thank Jefferson Baudin, Raymond Cheng, Zsolt Patakfalvi, Jinsong Xu for fruitful discussions. 

During the work on this paper, I was financially supported by grant \#200021-231484 from the Swiss National Science Foundation. 

\section{Preliminaries}\label{section: preliminaries}
\subsection{Tannakian subcategories in $\mathrm{Coh}(X)$}\label{subsection: Tannakian duality}
Let $k$ be a field. 
Recall that a neutral Tannakian category is a rigid $k$-linear tensor category $\calC$ with $\End(1_\calC)=k$ and a fibre functor $\calC \rightarrow \mathrm{Vec}_k$. 
They are equivalent to a category $\mathrm{Rep}_G$ of finite dimensional representations of an affine group scheme. 
This gives a one-to-one correspondence between neutral Tannakian categories and affine group schemes. 
We refer to \cite{deligneTannakianCategories} for a detailed discussion on the above equivalence. 
We are interested in subcategories in the category of coherent sheaves over a variety $X$.  
\begin{set}
	In this subsection, we work over a perfect field $k$ of characteristic $p\geq 0$ and a smooth projective variety $X$ over $k$. 
\end{set}

Let $P\rightarrow X$ be a principal $G$-bundle. 
Given a finite dimensional representation $V$ of $G$, one can form the associated fibre bundle $(P\times V)/G \rightarrow X$, which is a geometric vector bundle. 
This defines a functor $F: \mathrm{Rep}_G \rightarrow \mathrm{Coh}(X)$, whose image is a neutral Tannakian subcategory. 
Conversely, given a functor $F: \mathrm{Rep}_G \rightarrow \mathrm{Coh}(X)$, one can associate a principal $G$-bundle to it as follows: 
Take the regular representation of $G$ on $k[G]$ and consider the quasi-coherent sheaf $\calP:=\varinjlim_{V\subset k[G] \text{ fin. dim.}} F(V)$. 
Nori showed in \cite[Lemma 2.2 \& Lemma 2.3]{noriFundamentalGroupscheme1982} that $\calP$ has an $\calO_X$-algebra structure and the corresponding affine map $P\rightarrow X$ is a principal $G$-bundle. 
Nori also showed in \cite[Proposition 2.9]{noriFundamentalGroupscheme1982} that the two constructions above are inverse to each other. More precisely:

\begin{prop}\label{prop: correspondence principal bundle and tannakian}
	There is a one-to-one correspondence between functors from neutral Tannakian subcategories to $\mathrm{Coh}(X)$ and principal bundles on $X$ associated to affine group schemes, given by the above mentioned constructions. 
\end{prop}

The principal $G$-bundle $P\rightarrow X$ should be considered as a universal cover associated to the functor $F$. 
The following Corollary follows then directly from the constructions.

\begin{cor}\label{cor: pullback of vector bundle becomes trivial}
	Let $\calE$ be a vector bundle on $X$ sitting in a neutral Tannakian subcategory in $\mathrm{Coh}(X)$, and let $\pi:P \rightarrow X$ be the associated principal bundle. 
	Then $\pi^*\calE \cong \calO_P^{\oplus n}$, where $n=\rank \calE$. 
\end{cor}

Moreover, we can show that if $F$ is a full inclusion, then the scheme $P$ has few global sections. 

\begin{prop}\label{prop: rigidity of universal cover}
	Let $\pi:P\rightarrow X$ be the principal $G$-bundle constructed above corresponding to a full neutral Tannakian subcategory $F:\mathrm{Rep}_G \hookrightarrow \mathrm{Coh}(X)$, then $H^0(P,\calO_P)=k$. 
\end{prop}
\begin{proof}
	We have the following equalities:
	\begin{align*}
		&H^0(P,\calO_P)\\
		&= H^0(X,\pi_*\calO_P) \\
		&= H^0(X, \calP)
		&\text{(By definition)}\\
		&=H^0 \left( X,\textstyle{\varinjlim_{V\subset k[G] \text{ fin. dim.}}} F(V) \right)\\
		&= \textstyle{\varinjlim_{V\subset k[G] \text{ fin. dim.}}} H^0(X,F(V)) 
		&\text{(\cite[Tag 009F]{stacks-project})}\\
		&= \textstyle{\varinjlim_{V\subset k[G] \text{ fin. dim.}}} \Hom_{\calO_X}(\calO_X,F(V))\\
		&= \textstyle{\varinjlim_{V\subset k[G] \text{ fin. dim.}}} \Hom_{\calO_X}(F(k),F(V)) 
		&(\text{Note $\calO_X=F(k)$ with trivial $G$-action on $k$})\\
		&= \textstyle{\varinjlim_{V\subset k[G] \text{ fin. dim.}}} \Hom_{G}(k,V) 
		&(\text{$F$ is a full functor})\\
		&= \textstyle{\varinjlim_{V\subset k[G] \text{ fin. dim.}}} V^G\\
		&=k[G]^G 
		&(\text{\cite[Section 2.13]{jantzenRepresentationsAlgebraicGroups2014}})\\
		&= k.
	\end{align*}
\end{proof}

Let $\calC \subset \calD$ be an inclusion of neutral Tannakian subcategories in $\mathrm{Coh}(X)$, which is identified as an inclusion $\mathrm{Rep}_H \subset \mathrm{Rep}_G$. 
This inclusion corresponds to a group homomorphism $G\rightarrow H$ which is faithfully flat by \cite[Proposition 2.21]{deligneTannakianCategories}, and is identified as endowing every $H$-module a $G$-module structure along the group homomorphism. 
We can describe also the relations between the two principal bundles $P_\calC$ and $P_\calD$.

\begin{prop}[{\cite[Proposition 2.9]{noriFundamentalGroupscheme1982}}]
	We have $P_\calC= (P_\calD \times H)/G$, where $G$ acts on the product diagonally. 
\end{prop}

\begin{cor}\label{cor: associated fibre bundle comparison}
	Let $X' \rightarrow X$ be an associated fibre bundle to $P_{\calC}$, then it is also an associated fibre bundle to $P_{\calD}$. 
\end{cor}

We discuss then briefly how the structure of a neutral Tannakian category reflects the structure of the corresponding affine group scheme. 
These criteria work for any neutral Tannakian category, so we may shortly drop the assumption $\calC \subset \mathrm{Coh}(X)$. 
Let $\calE \in \calC$ be an object in a neutral Tannakian category. 
Consider the full subcategory $\langle \calE \rangle^\otimes$ consisting of all subquotients of some $f(\calE, \calE^\vee)$,  where $f\in \bbN[s,t]$ is a polynomial, and addition and multiplication are taken as direct sum and tensor product. 
It is then easy to check

\begin{lem}\label{lem: Tannakian subcategory generated by one element}
	The subcategory $\langle \calE \rangle^\otimes$ is again neutral Tannakian. 
\end{lem}

If $\langle \calE \rangle^\otimes=\calC$, or in other words, every bundle in $\calC$ is a subquotient of some $f(\calE, \calE^\vee)$, then we call $\calE$ a \textbf{tensor generator} of $\calC$. 
There is a criterion on the linear algebraicity of $G$.

\begin{prop}[{\cite[Proposition 2.20]{deligneTannakianCategories}}]\label{prop: algebraicity from Tannakian}
	Let $G$ be an affine group scheme corresponding to a neutral Tannakian category $\mathrm{Rep}_G$. Then $G$ is linear algebraic if and only if $\mathrm{Rep}_G$ admits a tensor generator. 
\end{prop}

\subsection{Numerical flatness and S-fundamental group scheme}
Let $X$ be a smooth projective variety over a perfect field $k$. 
A vector bundle $\calE$ on $X$ is called numerically flat if both $\calE$ and $\calE^\vee$ are nef. 
In \cite{langerSfundamentalGroupScheme2011}, it is shown that all numerically flat bundles on $X$ form a neutral Tannakian category $\mathrm{Vec}_0^S(X)$, and Langer defines the S-fundamental group scheme to be the corresponding Tannakian dual. 
We discuss here some useful facts that we require for proofs later. 
\begin{set}
	In this subsection, we work over a perfect field $k$ of characteristic $p\geq 0$ and a smooth projective variety $X$ over $k$.  
\end{set}

\begin{defn}
	A vector bundle $\calE$ on $X$ is \textbf{nef} if the tautological quotient $\calO_{\bbP_X(\calE)}(1)$ is nef on $\bbP_X(\calE)$. 
	It is called \textbf{numerically flat} if both $\calE$ and $\calE^\vee$ are nef. 
\end{defn}

\begin{prop}[{\cite[Proposition 5.1]{langerSfundamentalGroupScheme2011}}]\label{prop: characterization of numerical flatness}
	Let $\calE$ be a vector bundle on $X$. 
	The following conditions are equivalent:
	\begin{enumerate}
		\item $\calE$ is numerically flat,
		\item there exists an ample divisor $H$ such that $\calE$ is strongly $H$-semistable and $c_1(\calE)\cdot H^{n-1}=c_2(\calE)\cdot H^{n-2}=0$. 
		\item $\calE$ is nef of degree $0$ with respect to some ample divisor $H$.
	\end{enumerate}
\end{prop}

Here, strongly $H$-semistable means that all Frobenius pullbacks $F^{e,*}\calE$ are $H$-semistable.
Langer shows furthermore that given a point $x\in X(k)$, the full subcategory in $\mathrm{Coh}(X)$ consisting of numerically flat bundles together with the fibre functor $\calE \mapsto \calE|_x$ is a neutral Tannakian category.

\begin{defn}
	The \textbf{S-fundamental group scheme} $\pi_1^S(X,x)$ is the Tannakian dual to the full subcategory $\mathrm{Vec}_0^S(X)$ of $\mathrm{Coh}(X)$ consisting of numerically flat bundles. 
	The \textbf{S-universal cover} $\widetilde X^S$ is the principal $\pi_1^S(X,x)$-bundle given by the correspondence in Proposition \ref{prop: correspondence principal bundle and tannakian}. 
\end{defn}

\subsection{Grothendieck duality along finite morphisms}
We recall here the Grothendieck duality of finite morphisms between pairs. 
A reference for the results can be found in
\cite[Section 5.5]{kollarBirationalGeometryAlgebraic2002}.

\begin{set}
	In this subsection, we work over a field $k$. 
	Given a scheme $X$, we always assume that $X$ is of finite type over $k$.  
\end{set}
Let $f:X\rightarrow Y$ be a finite morphism between schemes, and let $\calF$ be a coherent sheaf on $X$, $\calG$ a coherent sheaf on $Y$. 
We define $f^!\calG:=\SheafHom_{\calO_Y}(\calO_X,\calG)$ equipped with its natural coherent $\calO_X$-module structure. 
The Grothendieck duality then reads
\begin{align*}
	\SheafHom_{\calO_Y}(f_*\calF, \calG) &\cong f_*\SheafHom_{\calO_X}(\calF, f^!\calG),\\
	\Hom_{\calO_Y}(f_*\calF, \calG) &\cong \Hom_{\calO_X}(\calF, f^!\calG).
\end{align*}
We define the relative dualizing sheaf $\omega_{X/Y}$ as $f^!\calO_Y$, and one can show that the formation of the relative dualizing sheaf is compatible with flat base changes on $Y$. 
In particular,

\begin{lem}\label{lem: relative dualizing sheaf of field extension}
	Let $Y$ be a scheme over an $F$-finite field $k$ and let $X:=Y\times_k k^{\frac{1}{p}}$. 
	Then $\omega_{X/Y}=\calO_X$.
\end{lem}

Assume now that $X$ is normal. 
Then there exists a reflexive sheaf $\omega_X$ which agrees with $\Omega_X^{\dim_X}$ on the smooth locus, called the dualizing sheaf. 
The canonical divisor $K_X$ is a Weil divisor on $X$ such that $\calO_X(K_X)\cong \omega_X$. 
If both $X$ and $Y$ are normal, then there is an isomorphism $\omega_{X/Y}\cong (\omega_X\otimes f^*\omega_Y^{\vee})^{\vee\vee}$, from which we may directly obtain the following:
\begin{lem}\label{lem: duality}
	Assume that $X$ and $Y$ are normal, and let $D$ be a divisor on $X$. 
	Then there is a one-to-one correspondence
	\[
		\{\varphi: f_*\calO_X(D) \rightarrow \calO_Y \text{ homomorphisms of $\calO_Y$-modules}\} 
		\cong 
		H^0(X,\calO_X(K_X-f^*K_Y-D))
	\]
	which is natural in $D$. 
\end{lem}

\subsection{$F$-singularities}
The notion of $F$-singularities for a pair with boundary was introduced by Hara and Watanabe in \cite{haraFregularFpureRings2001}. 
They should be considered as characteristic $p$ replacements for singularities arising from the minimal model program. 
We briefly recall the definitions and properties of different types of $F$-singularities that we shall use later. 

\begin{set}
	In this subsection, we work over an $F$-finite field $k$ of characteristic $p>0$. 
	By a pair $(X,\Delta)$, we always assume that $X$ is normal. 
	The \textbf{Cartier index} of $K_X+\Delta$ is the smallest positive integer $m$ such that $m(K_X+\Delta)$ is Cartier.
\end{set}

\begin{defn}
	We say that $(X,\Delta)$ is
	\begin{enumerate}
		\item \textbf{globally $F$-regular} if for any effective Weil divisor $D$, the natural $\calO_X$-module homomorphism $\calO_X \rightarrow F^e_* \calO_X (\lceil (p^e-1)\Delta \rceil + D)$ admits a splitting for some integer $e>0$,
		\item \textbf{globally purely $F$-regular} if for any effective Weil divisor $D$ whose support contains no component of $\lfloor \Delta \rfloor$, the natural $\calO_X$-module homomorphism $\calO_X \rightarrow F^e_* \calO_X (\lceil (p^e-1)\Delta \rceil + D)$ admits a splitting for some integer $e>0$,
		\item \textbf{globally $F$-split} if the natural $\calO_X$-module homomorphism $\calO_X \rightarrow F^e_* \calO_X( \lceil (p^e-1)\Delta \rceil )$ admits a splitting for some integer $e>0$. 
	\end{enumerate}
\end{defn}

There are also relative and local versions.

\begin{defn}
	Let $f: X \rightarrow Y$ be a morphism. 
	Let $X'$ be the base change of $X$ along the $e$-th Frobenius of $Y$, and $F^e_{X/Y}:X\rightarrow X'$ be the relative $e$-th Frobenius. 
	\begin{enumerate}
		\item \textbf{globally $F$-regular} (resp. \textbf{globally purely $F$-regular} resp. \textbf{globally $F$-split}) over $Y$ if we have the same splitting, but with the domain $\calO_X$ replaced by $\calO_{X'}$, and the pushforward $F^e_*$ replaced by $F^e_{X/Y,*}$,
		\item \textbf{strongly $F$-regular} (resp. \textbf{purely $F$-regular} resp. \textbf{sharply $F$-pure}) if there exists an open over $X=\bigcup_i X_i$ such that $(X_i, \Delta|_{X_i})$ is globally $F$-regular (resp. globally purely $F$-regular resp. globally $F$-split). 
	\end{enumerate}
\end{defn}

\begin{rmk}
	In particular, the absolute global $F$-regularity (resp. global pure $F$-regularity resp. global $F$-splitting) can be considered as relative 
	global $F$-regularity (resp. global pure $F$-regularity resp. global $F$-splitting) over $\bbF_p$.
\end{rmk}

We list here some fundamental properties of the above defined splitting notions. 
First, we give a criterion on global $F$-splitting using Grothendieck duality. 

\begin{lem}\label{lem: surjectivity criterion of GFS and GFR}
	The pair $(X,\Delta)$ is globally $F$-split if and only if the natural map $H^0(X,F_*^e\calO_X(\lfloor (1-p^e)(K_X+\Delta) \rfloor)) \rightarrow H^0(X,\calO_X)$ is surjective for some integer $e>0$. 
\end{lem}
\begin{proof}
	Applying the functor $\Hom_{\calO_X}(-,\calO_X)$ to the natural map $\calO_X\rightarrow F_*^e\calO_X( \lceil (p^e-1)\Delta \rceil )$, we see that global $F$-splitting is equivalent to the surjectivity of $\Hom_{\calO_X}(F_*^e\calO_X( \lceil (p^e-1)\Delta \rceil ),\calO_X) \rightarrow H^0(X,\calO_X)$ for some $e$. 
	Then apply Lemma \ref{lem: duality} to the left hand side.  
\end{proof}

Then we show that we have certain flexibilities on the exponent $e$ we may choose for a splitting to exist. 

\begin{lem}\label{lem: multiplies of GFS index are still GFS}
	Assume that there exists an integer $e$ such that the natural map $\calO_X \rightarrow F^e_* \calO_X( \lceil (p^e-1)\Delta \rceil )$ is split. 
	Then for any integer $n\geq 1$, the natural map $\calO_X \rightarrow F^{ne}_* \calO_X( \lceil (p^{ne}-1)\Delta \rceil )$ admits a splitting.
\end{lem}
\begin{proof}
	We proceed by an induction. 
	Assume $\calO_X \rightarrow F^{ke}_* \calO_X( \lceil (p^{ke}-1)\Delta \rceil )$ for some integer $k\geq 1$, then we may twist the map $\calO_X \rightarrow F^e_* \calO_X( \lceil (p^e-1)\Delta \rceil )$ by $\calO_X( \lceil (p^{ke}-1)\Delta \rceil )$ and take the pushforward by $F^{ke}$. 
	This gives us a split map 
	\[
		F^{ke}_*\calO_X( \lceil (p^{ke}-1)\Delta \rceil ) \rightarrow F^{(k+1)e}_* \calO_X( \lceil (p^e-1)\Delta \rceil + p^e \lceil (p^{ke}-1) \Delta \rceil).
	\] 
	Hence the composition 
	\[
		\calO_X \rightarrow F^{ke}_* \calO_X( \lceil (p^{ke}-1)\Delta \rceil ) \rightarrow F^{(k+1)e}_* \calO_X( \lceil (p^e-1)\Delta \rceil + p^e \lceil (p^{ke}-1) \Delta \rceil)
	\] 
	is also split. 
	The inequalities 
	\[
		\lceil (p^{(k+1)e}-1)\Delta \rceil
		\leq \lceil (p^e-1)\Delta \rceil + \lceil (p^{(k+1)e}-p^e) \Delta \rceil
		\leq \lceil (p^e-1)\Delta \rceil + p^e \lceil (p^{ke}-1) \Delta \rceil
	\] identify $\calO_X(\lceil (p^{(k+1)e}-1)\Delta \rceil)$ as a subsheaf of $\calO_X( \lceil (p^e-1)\Delta \rceil + p^e \lceil (p^{ke}-1) \Delta \rceil)$, sending the canonical section to the canonical section. 
	Hence $\calO_X \rightarrow F^{(k+1)e}_*\calO_X(\lceil (p^{(k+1)e}-1)\Delta \rceil)$ is split, finishing the induction. 
\end{proof}

Next, we show that relative global $F$-splittings are stable under base change, composition and descent, given some mild extra conditions. 

\begin{lem}\label{lem: base change of relative GFS}
	Let $(X,\Delta)$ be a pair, and $f:(X,\Delta) \rightarrow S$ be a morphism, such that $(X,\Delta)$ is globally $F$-split over $S$. 
	Then $(X,\Delta)_T$ is globally $F$-split over $T$ for any flat base change $T \rightarrow S$. 
\end{lem}
\begin{proof}
	Let $X'$ be the base change of $X$ along $F^e:S \rightarrow S$. 
	Then one can easily check that the following diagram is commutative and Cartesian:
	\[\begin{tikzcd}
		{X_T} & X \\
		{(X')_T\cong(X_T)'} & {X'.}
		\arrow[from=1-1, to=1-2]
		\arrow["{F^e_{X_T/T}}"', from=1-1, to=2-1]
		\arrow["{F^e_{X/S}}", from=1-2, to=2-2]
		\arrow["\pi"', from=2-1, to=2-2]
	\end{tikzcd}\]
	So by \cite[Tag 02KG]{stacks-project}, we have $\pi^*F^e_{X/S,*}\calO_X(\lceil (p^e-1)\Delta \rceil) \cong F^e_{X_T/T,*}\calO_{X_T}(\lceil (p^e-1)\Delta \rceil_T)$. 
	There is also an inclusion $F^e_{X_T/T,*}\calO_{X_T}(\lceil (p^e-1)\Delta_T \rceil) \subset F^e_{X_T/T,*}\calO_{X_T}(\lceil (p^e-1)\Delta \rceil_T)$, hence we can pullback a splitting of $\calO_{X'} \rightarrow F^e_{X/S,*}\calO_X(\lceil (p^e-1)\Delta \rceil)$ on $X'$ to get a splitting of $\calO_{(X_T)'} \rightarrow F^e_{X_T/T,*}\calO_{X_T}(\lceil (p^e-1)\Delta_T \rceil)$ on $(X_T)'$. 
\end{proof}

\begin{lem}\label{lem: compositum of relative GFS}
	Let $(X,\Delta)$ be a pair and $f:X\rightarrow Y, g:Y \rightarrow Z$ be two morphisms. 
	Assume
	\begin{enumerate}
		\item $(X,\Delta)$ is globally $F$-split over $Y$, and
		\item $Y$ is globally $F$-split over $Z$. 
	\end{enumerate}
	Then $(X,\Delta)$ is globally $F$-split over $Z$.
\end{lem}
\begin{proof}
	Consider the following diagram:
	\[\begin{tikzcd}
		X &&& \\
		& {X'} & {X''} & X \\
		& Y & {Y''} & {Y.}
		\arrow["{{F^e_{X/Y}}}"{description}, from=1-1, to=2-2]
		\arrow["{F^e_{X/Z}}", from=1-1, to=2-3]
		\arrow["f"', from=1-1, to=3-2]
		\arrow["\pi"', from=2-2, to=2-3]
		\arrow["{f'}", from=2-2, to=3-2]
		\arrow["\lrcorner"{anchor=center, pos=0.125}, draw=none, from=2-2, to=3-3]
		\arrow[from=2-3, to=2-4]
		\arrow["{f''}", from=2-3, to=3-3]
		\arrow["\lrcorner"{anchor=center, pos=0.125}, draw=none, from=2-3, to=3-4]
		\arrow["f", from=2-4, to=3-4]
		\arrow["{F^e_{Y/Z}}"', from=3-2, to=3-3]
		\arrow["{g^*F^e_Z}"', from=3-3, to=3-4]
	\end{tikzcd}\]
	where $X''$ and $Y''$ are base changes of $X$ and $Y$ along the $e$-th absolute Frobenius $F^e:Z \rightarrow Z$. 
	Using Lemma \ref{lem: multiplies of GFS index are still GFS}, we can assume the integer $e$ is so picked such that $\calO_{X'} \rightarrow F^e_{X/Y,*} \calO_X( \lceil (p^e-1)\Delta \rceil )$ and $\calO_{Y''} \rightarrow F^e_{Y/Z,*}\calO_Y$ are split.
	It is easy to check that $X'$ is indeed the base change of $X$ along $F^e:Y\rightarrow Y$. 
	By \cite[Tag 02KG]{stacks-project} there is an isomorphism $f''^*F^e_{Y/Z,*}\calO_Y\cong \pi_*\calO_{X'}$. 
	Hence the map $(\calO_{X''} \rightarrow \pi_*\calO_{X'}) \cong (f''^*\calO_{Y''} \rightarrow f''^*F^e_{Y/Z,*}\calO_Y)$ is split. 
	Now $\calO_{X''} \rightarrow F^e_{X/Z,*} \calO_X( \lceil (p^e-1)\Delta \rceil )$ can be written as $\calO_{X''} \rightarrow \pi_*\calO_{X'} \rightarrow \pi_*F^e_{X/Y,*} \calO_X( \lceil (p^e-1)\Delta \rceil )$, the composition of a split map with the pushforward of a split map, hence is also split. 
\end{proof}

\begin{lem}\label{lem: descent of relative GFS}
	Let $(X,\Delta)$ be a pair and $f:X\rightarrow Y, g:Y \rightarrow Z$ be two morphisms. 
	Assume
	\begin{enumerate}
		\item $(X,\Delta)$ is globally $F$-split over $Z$, 
		\item $f_*\calO_X=\calO_Y$, and
		\item $Z$ is smooth over $\bbF_p$. 
	\end{enumerate}
	Then $Y$ is globally $F$-split over $Z$. 
\end{lem}
\begin{proof}
	Consider the same diagram as in Lemma \ref{lem: compositum of relative GFS}. 
	Since $F^e_Z$ is flat by Kunz Theorem (\cite[Tag 0EC0]{stacks-project}), applying flat base change \cite[02KH]{stacks-project} to the right square and $\calO_X$ yields $f''_*\calO_{X''}\cong \calO_{Y''}$. 
	The map $\calO_{X''} \rightarrow \pi_*\calO_X'$ is split as it factors the split map $\calO_{X''} \rightarrow F^e_{X/Z,*} \calO_X( \lceil (p^e-1)\Delta \rceil )$, and pushing it forward gives a map $\calO_{Y''} \rightarrow F^e_{Y/Z,*}f'_*\calO_{X'}$ which factors through $\calO_{Y''} \rightarrow F^e_{Y/Z,*}\calO_Y$, hence the last map is split. 
\end{proof}

\begin{rmk}
	In general, the global $F$-splitting does not satisfy a cancellation law. 
	Namely, $(X,\Delta)$ being absolutely globally $F$-split does not imply that $(X,\Delta)$ is relatively globally $F$-split over a base $Y$. 
	We need more constraints to make the statement true, see Proposition \ref{prop: GFS is stable under base change}. 
\end{rmk}

The global $F$-splitting can be considered as a characteristic $p$ analogue of being log Calabi-Yau. 

\begin{lem}\label{lem: adding extra boundary to GFS pair}
	Assume that there exists an integer $e$ such that the natural map $\calO_X \rightarrow F^e_* \calO_X( \lceil (p^e-1)\Delta \rceil )$ is split. 
	Then there exists an effective Weil divisor $\Gamma$ such that $\Gamma \sim \lfloor (1-p^e)(K_X+\Delta) \rfloor$ and $(X,\Delta+\frac{1}{p^e-1}\Gamma)$ is globally $F$-split. 
\end{lem}
\begin{proof}
	Let $F^e_* \calO_X( \lceil (p^e-1)\Delta \rceil ) \rightarrow \calO_X$ be the map that splits $\calO_X \rightarrow F^e_* \calO_X( \lceil (p^e-1)\Delta \rceil )$. 
	By Lemma \ref{lem: duality}, there is a divisor $\Gamma$ corresponding to the section $\calO_X \rightarrow \calO_X( \lfloor (1-p^e)(\Delta+K_X)\rfloor)$. 
	We can perform the transformations in Lemma \ref{lem: duality} again to the identity map $\calO_X \rightarrow \calO_X( \lfloor (1-p^e)(\frac{1}{p^e-1}\Gamma+\Delta+K_X)\rfloor) \cong \calO_X$ in the reverse direction and see that the induced map $F^e_* \calO_X( \lceil (p^e-1)(\Delta+\frac{1}{p^e-1}\Gamma) \rceil ) \rightarrow \calO_X$ splits the natural map $\calO_X \rightarrow F^e_* \calO_X( \lceil (p^e-1)(\Delta+\frac{1}{p^e-1}\Gamma) \rceil )$. 
\end{proof}

\begin{cor}\label{cor: GFS is log CY}
	If the pair $(X,\Delta)$ is globally $F$-split and the Cartier index of $K_X+\Delta$ is coprime to $p$, then there exists an effective $\bbQ$-Cartier $\bbQ$-divisor $\Delta'$ such that
	\begin{enumerate}
		\item $(X,\Delta+\Delta')$ is globally $F$-split, 
		\item $\Delta'$ and $K_X+\Delta+\Delta'$ have Cartier index coprime to $p$, 
		\item $K_X+\Delta+\Delta'\sim_\bbQ 0$.
	\end{enumerate} 
\end{cor}
\begin{proof}
	Let $m$ denote the Cartier index of $K_X+\Delta$. 
	Assume that the natural map $\calO_X \rightarrow F^e_* \calO_X( \lceil (p^e-1)\Delta \rceil )$ is split for some $e$. 
	Since $p^e$ is coprime to $m$, there exists some positive integer $k$ such that $m$ divides $p^{ke}-1$. 
	By Lemma \ref{lem: multiplies of GFS index are still GFS}, the natural map $\calO_X \rightarrow F^{ke}_* \calO_X( \lceil (p^{ke}-1)\Delta \rceil )$ is still split. 
	Then we apply Lemma \ref{lem: adding extra boundary to GFS pair} and find an effective Weil divisor $\Gamma$ such that $\Gamma \sim \lfloor (1-p^{ke})(K_X+\Delta) \rfloor = (1-p^{ke})(K_X+\Delta)$, and $(X,\Delta+\frac{1}{p^{ke}-1}\Gamma)$ is globally $F$-split. 
	So we may define $\Delta'=\frac{1}{p^{ke}-1}\Gamma$.
\end{proof}

Similarly, the strong $F$-regularity (resp. pure $F$-regularity) can be considered as a characteristic $p$ analogue of klt (resp. dlt) singularity. 

\begin{lem}[{\cite[Theorem 3.3]{haraFregularFpureRings2001}}]
	If $(X,\Delta)$ is purely $F$-regular, then $(X,\Delta)$ is dlt. 
\end{lem}

\begin{lem}[{\cite[Proposition 2.2]{haraFregularFpureRings2001}}]
	The pair $(X,\Delta)$ is strongly $F$-regular if and only if $(X,\Delta)$ is purely $F$-regular and $\lfloor\Delta\rfloor=0$.
\end{lem}

\begin{lem}[{\cite[Proposition 2.2]{haraFregularFpureRings2001}}]\label{lem: SFR with smaller boundary}
	If $(X,\Delta)$ is strongly $F$-regular (resp. purely $F$-regular), then so is $(X,\Delta')$ for any $\bbQ$-effective divisor $\Delta'\leq \Delta$. 
\end{lem}

Combining the last two Lemmas, we may directly obtain the following Corollary we need. 

\begin{cor}\label{cor: midpoint of SFR and PFR is SFR}
	If $(X,\Delta)$ is strongly $F$-regular and $\Gamma$ is an integral Weil divisor such that $(X,\Delta+\Gamma)$ is purely $F$-regular, then $(X,\Delta+\varepsilon\Gamma)$ is strongly $F$-regular for all $\varepsilon\in [0,1)\cap\bbQ$.
\end{cor}

\section{Criterion of being associated fibre bundles}\label{section: Criterion of being associated fibre bundles}
In this section we proceed to prove Theorem \ref{thm: main theorem} and see its applications to the Albanese morphism, which leads to Theorem \ref{thm: associated bundle abelian variety positive characteristic}. 

\begin{set}
	Throughout this section, we work over a perfect field of characteristic $p\geq 0$, or over $\bbC$ when specified. 
\end{set}

\subsection{Criterion of associated bundle}
Let $f:(X,\Delta) \rightarrow Y$ be a projective morphism, and let $\calL$ be an $f$-very ample line bundle on $X$. 
The datum of $f$ is completely encoded in the graded section algebra $\bigoplus_{m=0}^\infty f_*(\calL^{\otimes m})$, and the multiplication rule in the section algebra can be interpreted as the map $\Sym^m f_*\calL \rightarrow f_*(\calL^{\otimes m})$. 
The following lemma is then obvious. 

\begin{lem}\label{lem: triviality of relative ample bundle implies product}
	Let $y\in Y(k)$ be a distinguished point, and let $(X_y, \Delta_y)$ be the fibre of $f$ over $y$. The pair $(X,\Delta)$ is isomorphic to the product $Y \times (X_y,\Delta_y)$ over $Y$, if for some $d\gg 0$ and for every irreducible component $D$ of $\Delta$, there is a positive integer $r$ and isomorphisms of graded $\calO_Y$-algebras
	\begin{align*}
		\bigoplus_{m\geq 0} f_*(\calL^{\otimes dm}) &\cong \bigoplus_{m\geq 0} \calO_Y \otimes_k H^0(X_y,\calL_y^{\otimes dm}),\\
		\bigoplus_{m\geq 0} f_*(\calL^{\otimes dm}(-rD)) &\cong \bigoplus_{m\geq 0} \calO_Y \otimes_k H^0(X_y,\calL_y^{\otimes dm}(-rD_y))
	\end{align*}
	that commute with the obvious inclusions on both sides, or equivalently, if there is an integer $d\gg 0$, such that for every irreducible component $D$ of $\Delta$, there is a positive integer $r$ and isomorphisms
	\[\begin{tikzcd}
		{\calO_Y \otimes_k \Sym^{dm} H^0(X_y,\calL_y(-rD_y))} & {\Sym^{dm} f_*(\calL(-rD))} && \\
		{\calO_Y \otimes_k H^0(X_y,\calL_y^{\otimes dm}(-rD_y))} & {f_*(\calL^{\otimes dm}(-rD))} \\
		& {\calO_Y \otimes_k \Sym^{dm} H^0(X_y,\calL_y)} && {\Sym^{dm} f_*\calL } \\
		& {\calO_Y \otimes_k  H^0(X_y,\calL_y^{\otimes dm})} && {f_*(\calL^{\otimes dm})}
		\arrow["\cong"{description}, draw=none, from=1-1, to=1-2]
		\arrow[from=1-1, to=2-1]
		\arrow[from=1-1, to=3-2]
		\arrow[from=1-2, to=2-2]
		\arrow[from=1-2, to=3-4]
		\arrow["\cong"{description}, draw=none, from=2-1, to=2-2]
		\arrow[from=2-1, to=4-2]
		\arrow[from=2-2, to=4-4]
		\arrow["\cong"{description}, draw=none, from=3-2, to=3-4]
		\arrow[from=3-2, to=4-2]
		\arrow[from=3-4, to=4-4]
		\arrow["\cong"{description}, draw=none, from=4-2, to=4-4]
	\end{tikzcd}\]
	that make the cube commute for all $m\geq 0$. 
\end{lem}

\begin{rmk}
	It seems more natural to take $r=1$, but an arbitrary $r>0$ satisfying the condition will fix an infinitesimal neighbourhood of $D$ and hence the reduced $D$ itself. 
	Allowing such an $r>0$ would simplify the treatment on the bundles.
\end{rmk}

\begin{lem}\label{lem: rigidity of matrix}
	In the situation above, we assume further $H^0(Y,\calO_Y)=k$. 
	Then the cube automatically commutes. 
\end{lem}
\begin{proof}
	We note that, when restricting the cube to the fibre at $y$, the restricted cube commutes as all the horizontal isomorphisms become identity. 
	Then the lemma follows from the following claim that is easy to prove: 
	\begin{claim}
		Let $Y$ be a scheme over $k$ such that $H^0(Y,\calO_Y)=k$, and let $y\in Y(k)$ be a point. 
	Let $\varphi,\psi: \calO_Y^{\oplus n} \rightarrow \calO_Y^{\oplus m}$ be two maps of $\calO_Y$-modules such that $\varphi|_y=\psi|_y$, then $\varphi=\psi$.
	\end{claim}
\end{proof}

Let now $Y$ be smooth projective. 
Assume that $\calL$ is an $f$-ample line bundle on $X$, such that
\begin{equation}
\tag*{(\#)}
\begin{minipage}{0.9\textwidth}
\begin{enumerate}
	\item $f_*(\calL^{\otimes m})$ are numerically flat for all $m>0$, 
	\item the multiplication map $\Sym^m f_*\calL \rightarrow f_*(\calL^{\otimes m})$ is surjective for all $m> 0$,
\end{enumerate}
\end{minipage}
\end{equation}
and for every irreducible component $D$ of $\Delta$, there is a positive integer $r$ such that
\begin{equation}
\tag*{(\#)} \label{property very very ample}
\begin{minipage}{0.9\textwidth}
    \begin{enumerate}
	\item[(3)] $\calL(-rD)$ is $f$-ample, and
	\item[(4)] $f_*(\calL^{\otimes m}(-rD))$ are numerically flat for all $m>0$. 
	\end{enumerate}
\end{minipage}
\end{equation} 
We denote the collection of the properties above with \ref{property very very ample}. 
Let $\langle f_*\calL \rangle^\otimes$ be the full subcategory in $\mathrm{Vec}_0^S(Y)$ generated by $f_*\calL$ in the sense of Lemma \ref{lem: Tannakian subcategory generated by one element}. 

\begin{lem}\label{lem: Tannakian subcategory generated by f_*L contains all}
	The category $\langle f_*\calL \rangle^\otimes$ is neutral Tannakian. 
	If $\calL$ has property \ref{property very very ample}, then $\langle f_*\calL \rangle^\otimes$ contains $\Sym^m f_*\calL,\  f_*(\calL^{\otimes m}),\ \Sym^m f_*(\calL(-rD))$ and $f_*(\calL^{\otimes m}(-rD))$ for all $m> 0$ and all irreducible components $D$ of $\Delta$. 
\end{lem}
\begin{proof}
	The first statement is just Lemma \ref{lem: Tannakian subcategory generated by one element}. 
	Since the symmetric power is the cokernel of a map between tensor products, we have $\Sym^m f_*\calL \in \langle f_*\calL \rangle^\otimes$. 
	Then $f_*(\calL^{\otimes m})$ is a numerically flat quotient of $\Sym^m f_*\calL$, and $\Sym^m f_*(\calL(-rD))$ is a numerically flat subbundle of $\Sym^m f_*\calL$, and eventually $f_*(\calL^{\otimes m}(-rD))$ is a numerically flat subbundle of $f_*(\calL^{\otimes m})$. 
\end{proof}

We are ready to prove the following criterion of being an associated bundle. 

\begin{thm}\label{thm: main theorem in text}
	Let $(X,\Delta)$ be a projective pair and $Y$ be a smooth projective variety over a perfect field $k$. 
	Let $f:(X,\Delta)\rightarrow Y$ be a morphism, and let $y\in Y(k)$ be a $k$-point. 
	We denote with $(X_y,\Delta_y)$ the fibre of $f$ over $y$. 
	Assume there exists an $f$-ample line bundle $\calL$, such that
	\begin{enumerate}
		\item the sheaves $f_*(\calL^{\otimes m})$ are numerically flat for all $m>0$,
	\end{enumerate} 
	and for every irreducible component $D$ of $\Delta$, there is a positive integer $r$ such that
	\begin{enumerate}
		\item[(2)] $f_*(\calL^{\otimes m}(-rD))$ are numerically flat for all $m>0$. 
	\end{enumerate}
	Then we can find a linear algebraic group $G$ which is a quotient of $\pi_1^S(Y,y)$, a principal $G$-bundle $P\rightarrow Y$, such that $f$ is an $(X_y,\Delta_y)$-bundle associated to $P \rightarrow Y$. 
	In particular, $f$ is an $(X_y,\Delta_y)$-bundle associated to $\widetilde Y^S \rightarrow Y$.
\end{thm}
\begin{proof}
	By replacing $\calL$ with a large enough multiple we may assume that $\calL$ has property \ref{property very very ample}. 
	Let $G$ be the group scheme corresponding to the neutral Tannakian category $\langle f_*\calL \rangle^\otimes$, and let $\pi: P\rightarrow Y$ be the principal $G$-bundle given by Proposition \ref{prop: correspondence principal bundle and tannakian}. 
	Since $\langle f_*\calL \rangle^\otimes$ is a full neutral Tannakian subcategory of $\mathrm{Vec}_0^S(Y)$, the natural map $\pi_1^S(Y,y) \rightarrow G$ is faithfully flat by \cite[Proposition 2.21]{deligneTannakianCategories}. 
	The group scheme $G$ is linear algebraic by Proposition \ref{prop: algebraicity from Tannakian}.
	We name several maps as in the following Cartesian diagram. 
	\[\begin{tikzcd}
	{(X',\Delta')} & {(X,\Delta)} \\
		P & Y
		\arrow["{\pi'}", from=1-1, to=1-2]
		\arrow["{f'}"', from=1-1, to=2-1]
		\arrow["f", from=1-2, to=2-2]
		\arrow["\pi"', from=2-1, to=2-2]
	\end{tikzcd}\]
	By flat base change (\cite[Tag 02KH]{stacks-project}), we have $\pi^*f_*=\pi'^*f'_*$ as functors $\mathrm{QCoh}(X) \rightarrow \mathrm{QCoh}(P)$. 
	Denote $\calL':= \pi'^* \calL$ and $D':=D \times_Y P$. 
	Then there are isomorphisms
	\begin{align*}
		\pi^*\Sym^m f_*\calL &\cong \Sym^m f'_*\calL', \\
		\pi^*f_*(\calL^{\otimes m}) &\cong f'_*(\calL'^{\otimes m}), \\
		\pi^*\Sym^m f_*(\calL(-rD)) &\cong \Sym^m f'_*(\calL'(-rD')), \\
		\pi^*f_*(\calL^{\otimes m}(-rD)) &\cong f'_*(\calL'^{\otimes m}(-rD')). 
	\end{align*}
	By Corollary \ref{cor: pullback of vector bundle becomes trivial} and Lemma \ref{lem: Tannakian subcategory generated by f_*L contains all}, the four bundles on the left hand side are trivial, and by Proposition \ref{prop: rigidity of universal cover}, we have $H^0(P,\calO_P)=k$.
	This allows us to use Lemma \ref{lem: triviality of relative ample bundle implies product} and Lemma \ref{lem: rigidity of matrix} and conclude that $(X', \Delta') \cong P \times (X_y, \Delta_y)$ over $P$. 
	
	It still remains to show that the action of $G$ on $P \times (X_y, \Delta_y)$ is diagonal. 
	But indeed we have an embedding $P \times X_y \rightarrow \bbP_{P}(f'_*\calL') \cong P\times \bbP^N$ given by the surjection $\Sym^m f'_*\calL' \rightarrow f'_*(\calL'^{\otimes m})$. 
	And by construction in Section \ref{subsection: Tannakian duality}, $G$ acts diagonally on $P\times \bbP^N$, hence also diagonally on $P \times (X_y, \Delta_y)$. 
	This shows that $f:(X,\Delta) \rightarrow Y$ is an $(X_y,\Delta_y)$-bundle associated to $\pi: P \rightarrow Y$. 
	To show that $f$ is also a fibre bundle associated to the S-universal cover, simply apply Corollary \ref{cor: associated fibre bundle comparison}. 
\end{proof}

\subsection{Semi-positivity engine}
It seems a priori a very strong condition to have an $f$-ample line bundle $\calL$ such that $f_*(\calL^{\otimes m})$ are numerically flat for all $m \geq 0$. 
In positive characteristics, Ejiri shows however in \cite{ejiriSplittingAlgebraicFiber2023a} that this is always achievable when the relative anti-log-canonical divisor is nef with some mild extra conditions. 

\begin{thm}[{\cite[Theorem 4.2]{ejiriSplittingAlgebraicFiber2023a}}]\label{thm: existence of relative ample numerically flat bundle}
	Let $(X,\Delta)$ be a Cohen-Macaulay projective pair, and $Y$ be a smooth projective variety over a perfect field $k$ of characteristic $p>0$. 
	Given a surjective morphism $f:(X,\Delta) \rightarrow Y$ and assume also
	\begin{enumerate}
		\item the relative anti-log-canonical divisor $-K_{X/Y}-\Delta$ is nef and of Cartier index coprime to $p$,
		\item the geometric generic fibre $(X_{\overline\eta}, \Delta_{\overline\eta})$ is strongly $F$-regular.
	\end{enumerate}
	Then there exists an $f$-ample line bundle $\calL$ such that $f_*(\calL^{\otimes m})$ are numerically flat for all $m\geq 0$. 
\end{thm}

This theorem says nothing about the boundary. 
We use the ideas in \cite{patakfalviWeakBeauvilleBogomolovDecomposition2025} relying on semi-positivity engine to show that $f_*(\calL^{\otimes m}(-rD))$ are also numerically flat for all $m\gg 0$ and all irreducible components $D$ of $\Delta$, with $r$ chosen as the Cartier index of $D$. 
To this end, we apply the following settings.
\begin{set}\label{set: conditions for proving numerical flatness of ideal sheaves}
	We take the same settings and notations in Theorem \ref{thm: existence of relative ample numerically flat bundle}. 
	We remark also that by \cite[Theorem B]{patakfalviFsingularitiesFamilies2018} the strong $F$-regularity of the geometric generic fibre $(X_{\overline\eta}, \Delta_{\overline\eta})$ is equivalent to the strong $F$-regularity of the fibre over a general perfect point. 
	Here, a perfect point means just a map $\Spec l \rightarrow Y$ from the spectrum of a perfect field.
	Furthermore, let $\iota:C \rightarrow Y$ be a morphism from a smooth projective curve $C$ to $Y$. 
	We assume that there is a rational point $y\in C(k)$, and assume that the fibre $(X_y,\Delta_y)$ is strongly $F$-regular.  
	We denote $(Z,\Lambda):=(X,\Delta)\times_Y C,\ D':=D \times_Y C,\ \calL':= \calL|_Z$. 
	Write also $f'=f|_Z: (Z,\Lambda) \rightarrow C$. 
	Moreover, let $r$ be an integer such that $rD$ is Cartier. 
\end{set}

\begin{lem}\label{lem: miracle flatness}
	Under the above setting, $f$ is flat, and $Z$ is Cohen-Macaulay and normal. 
\end{lem}
\begin{proof}
	By \cite[Theorem 4.1]{ejiriSplittingAlgebraicFiber2023a}, the map $f$ is flat, and all fibres are reduced. 
	Any fibre of $f$ is Cohen-Macaulay by \cite[Tag 045J]{stacks-project}, and by the same cited argument, $Z$ is Cohen-Macaulay. 
	As $(X_y,\Delta_y)$ is strongly $F$-regular, the geometric generic fibre of $f'$ is strongly $F$-regular too (\cite[Theorem B]{patakfalviFsingularitiesFamilies2018}), and is in particular normal. 
	The normality of the geometric generic fibre together with the reducedness of all fibres imply then $Z$ is regular in codimension one. 
	By Serre's criterion on normality, $Z$ is normal. 
\end{proof}

Thus, the relative canonical divisor $K_{Z/C}$ is well-defined, and we have $K_{Z/C}+\Lambda \sim_\bbQ (K_{X/Y}+\Delta)|_C$. 
In particular, $-K_{Z/C}-\Lambda$ is nef.

\begin{cor}\label{cor: comparison of higher direct images}
	Let $m$ be an integer such that $R^if_*(\calL^{\otimes m}(-rD))=0$ for all $i>0$. 
	Then $f_*(\calL^{\otimes m}(-rD))|_C \cong f'_*(\calL'^{\otimes m}(-rD'))$, and $R^if'_*(\calL'^{\otimes m}(-rD'))=0$ for all $i>0$. 
\end{cor}
\begin{proof}
	Consider the base change diagram
		\[\begin{tikzcd}
			Z & X \\
			C & Y.
			\arrow["{\iota'}", from=1-1, to=1-2]
			\arrow["{f'}"', from=1-1, to=2-1]
			\arrow["f", from=1-2, to=2-2]
			\arrow["\iota"', from=2-1, to=2-2]
		\end{tikzcd}\]
		By derived base change (\cite[Tag 08IB]{stacks-project}), there is an isomorphism of complexes
		\[ 
			L\iota^* Rf_* (\calL^{\otimes m}(-rD)) \cong Rf'_* L\iota'^* (\calL^{\otimes m}(-rD)).
		\] 
		Since $X$ is flat over $Y$ by Lemma \ref{lem: miracle flatness}, we have $\calL^{\otimes m}(-rD)$ is flat over $Y$ too. 
		Then both sides read further
		\begin{align*}
			L\iota^* Rf_* (\calL^{\otimes m}(-rD)) &\cong L\iota^* f_* (\calL^{\otimes m}(-rD))\\
			&\cong \iota^* f_* (\calL^{\otimes m}(-rD)) &(f_*(\calL^{\otimes m}(-rD)) \text{ is locally free})\\
			Rf'_* L\iota'^* (\calL^{\otimes m}(-rD)) 
			&\cong Rf'_* \iota'^* (\calL^{\otimes m}(-rD)) &(\calL^{\otimes m}(-rD) \text{ is flat over $Y$})\\
			&\cong Rf'_* (\calL'^{\otimes m}(-rD')). 
		\end{align*}
		Since the left hand side is a complex concentrated in degree $0$, both claims follow.
\end{proof}

\begin{lem}\label{lem: top intersection vanishes}
	The top self-intersection $(\calL')^{\dim Z}=0$. 
\end{lem}
\begin{proof}
	We apply the theorem of Hirzebruch-Riemann-Roch (\cite[Theorem A.4.1]{hartshorneAlgebraicGeometry1977}) to $f'_*(\calL'^{\otimes m})$, which reads
	\[
		\chi(C,f'_*(\calL'^{\otimes m}))= \deg c_1(f'_*(\calL'^{\otimes m})) + \rank f'_*(\calL'^{\otimes m}) \cdot (1-g(C))
	\]
	By Proposition \ref{prop: characterization of numerical flatness}, we have $\deg c_1(f'_*(\calL'^{\otimes m}))=0$. 
	And we have also $\rank f'_*(\calL'^{\otimes m})=h^0(X_y,\calL_y^{\otimes m})=\chi(X_y,\calL_y^{\otimes m})$ for $m\gg 0$. 
	So asymptotically, $\chi(C,f'_*(\calL'^{\otimes m}))= O(m^{\dim X_y})= O(m^{\dim Z -1})$. 
	Meanwhile, we know $\chi(C,f'_*(\calL'^{\otimes m})) \sim \chi(Z,\calL'^{\otimes m}) \sim (\calL')^{\dim Z} \cdot m^{\dim Z}$, where the first asymptotic equivalence is because $\calL'$ is $f'$-ample, and the second is asymptotic Riemann-Roch (\cite[Tag 0BJ8]{stacks-project}). Hence we must have $(\calL')^{\dim Z}=0$. 
\end{proof}

\begin{cor}\label{cor: nefness of relative ample bundle}
	The line bundle $\calL'$ is nef. 
\end{cor}
\begin{proof}
	The proof is totally identical to \cite[Theorem 5.4]{patakfalviWeakBeauvilleBogomolovDecomposition2025}. 
	We simply remark here that the loc. cit. settings are more restricted on $\calL'$, but the proof essentially only uses that $\calL'$ is $f$-ample and the top intersection vanishes. 
\end{proof}

Recall the result on semi-positivity in \cite{patakfalviWeakBeauvilleBogomolovDecomposition2025}. 

\begin{thm}[{\cite[Theorem 3.1]{patakfalviWeakBeauvilleBogomolovDecomposition2025}}]\label{thm: semi-positivity}
	Under Settings \ref{set: conditions for proving numerical flatness of ideal sheaves}, let $E$ be a nef $\bbQ$-Cartier divisor on $Z$ such that $K_{Z/C}+\Lambda+E$ is $f'$-nef. 
	Then $K_{Z/C}+\Lambda+E$ is pseudo-effective and nef. 
\end{thm}

\begin{lem}\label{lem: nefness of ideal sheaf}
	Let $m$ be an integer such that $\calL'^{\otimes m}(-rD')$ is $f'$-ample and $\frac{r}{m}\leq \mathrm{coeff}_\Delta (D)$, then $\calL'^{\otimes m}(-rD')$ is also nef.
\end{lem}
\begin{proof}
	Let $L'$ be a divisor on $Z$ corresponding to $\calL'$.
	By Lemma \ref{lem: SFR with smaller boundary}, $(X_y,\Delta_y-\frac{r}{m}D_y)$ is also strongly $F$-regular. 
	Hence by the openness of strong $F$-regularity (\cite[Theorem B]{patakfalviFsingularitiesFamilies2018}), $(X_c,\Delta_c-\frac{r}{m}D_c)$ is also strongly $F$-regular for a general perfect point $c\in C$. 
	So we can apply Theorem \ref{thm: semi-positivity} as follows:
	\begin{align*}
		L'-\frac{r}{m}D' = 
			\underbrace{
				\underbrace{K_{Z/C}+\Lambda-\frac{r}{m}D'}_{\parbox{90pt}{\scriptsize $\left(X_c,\Delta_c-\frac{r}{m}D_c\right)$ is strongly $F$-regular for a general $c \in C$}} 
				+ \underbrace{(-K_{Z/C}-\Lambda) + L'}_{\text{nef}}
			}_{\text{$f'$-ample}},
	\end{align*}
	which shows that $L'-\frac{r}{m}D'$ is nef. 
\end{proof}

Let $\calN$ be a line bundle on $X_y$. 
For the next Proposition, we need a technical term $S^0(X_y,\Delta_y,\calN)$, defined as follows (see also \cite[Definition 2.5]{patakfalviSemipositivityPositiveCharacteristics2014}): 
Let $g$ be the minimal positive integer such that $(p^g-1)(K_{X_y}+\Delta_y)$ is Cartier. 
Let $\varphi_e$ be the natural map $F_*^{eg}\calO_{X_y}((1-p^{eg})(K_{X_y}+\Delta_y)) \rightarrow \calO_{X_y}$ as described in Lemma \ref{lem: surjectivity criterion of GFS and GFR}. 
Then we denote
\[
	S^0(X_y,\Delta_y,\calN):=\bigcap_{e\in \bbZ^{\geq 0}}\mathrm{Im}
	(H^0(X_y, F_*^{eg}\calO_{X_y}((1-p^{eg})(K_{X_y}+\Delta_y)) \otimes_{\calO_{X_y}} \calN )
	\longrightarrow
	H^0(X_y,\calN)).
\]

\begin{prop}\label{prop: numerical flatness of ideal sheaf}
	Let $m$ be an integer such that
	\begin{enumerate}
		\item $\calL'^{\otimes m}(-rD')$ is $f'$-ample, 
		\item $\frac{r}{m}\leq \mathrm{coeff}_\Delta (D)$,  
		\item $R^if'_* (\calL'^{\otimes m}(-rD'))=0$ for all $i>0$, and
		\item $S^0(X_y,\Delta_y, \calL_y^{\otimes m}(-rD_y))=H^0(X_y,\calL_y^{\otimes m}(-rD_y))$.
	\end{enumerate}
	Then $f'_*(\calL'^{\otimes m}(-rD'))$ is numerically flat on $C$. 
\end{prop}
\begin{proof}
	By Lemma \ref{lem: nefness of ideal sheaf}, we know that $\calL'^{\otimes m}(-rD')$ is nef. 
	Let $L'$ be a divisor corresponding to $\calL'$. 
	We have the following equality:
	\begin{align*}
		mL'-rD'= K_{Z/C}+\Lambda +
		\underbrace{(-K_{Z/C}-\Lambda) + mL'-rD'}_{\text{$f'$-ample and nef}},
	\end{align*}
	which allows us to apply \cite[Proposition 3.6]{patakfalviSemipositivityPositiveCharacteristics2014} and deduce that $f'_*(\calL'^{\otimes m}(-rD'))$ is nef. 
	Since there is an inclusion $f'_*(\calL'^{\otimes m}(-rD')) \hookrightarrow f'_*(\calL'^{\otimes m})$ and $f'_*(\calL'^{\otimes m})$ is semistable of degree $0$ by Proposition \ref{prop: characterization of numerical flatness}, we have $\deg(f'_*(\calL'^{\otimes m}(-rD')))\leq 0$. 
	Meanwhile, as $f'_*(\calL'^{\otimes m}(-rD'))$ is nef, we have also $\deg(f'_*(\calL'^{\otimes m}(-rD')))\geq 0$. 
	This shows that $f'_*(\calL'^{\otimes m}(-rD'))$ is nef of degree $0$, hence numerically flat by Proposition \ref{prop: characterization of numerical flatness}. 
\end{proof}

\begin{rmk}
	We remark here that for fixed $D,\calL,r$ and $y$, there exists an integer $m_0$, such that the conditions in Proposition \ref{prop: numerical flatness of ideal sheaf} hold for all $m\geq m_0$ and all curves $C$ going through $y$.  
	Indeed, the first two conditions hold obviously for all $m\gg 0$. 
	For the third condition, it suffices to use Corollary \ref{cor: comparison of higher direct images} and ask for the vanishing of $R^if_*(\calL^{\otimes m}(-rD))$ on $Y$. 
	Patakfalvi proved that the fourth condition holds for $m\gg 0$, see \cite[Lemma 2.20 \& Proposition 2.23]{patakfalviSemipositivityPositiveCharacteristics2014}.
\end{rmk}

\begin{thm}\label{thm: associated bundle relative anti-nef positive characteristic in text}
	Let $f:(X,\Delta) \rightarrow Y$ be a surjective fibration over a perfect field $k$ of characteristic $p> 0$, such that
	\begin{enumerate}
		\item $(X,\Delta)$ is a Cohen-Macaulay projective pair, 
		\item $Y$ is a smooth projective variety with a $k$-rational point $y\in Y(k)$, 
		\item the geometric generic fibre $(X_{\overline \eta}, \Delta_{\overline \eta})$ is strongly $F$-regular,
		\item $-K_{X/Y}-\Delta$ is nef and of Cartier index coprime to $p$.
	\end{enumerate}
	Then there exists a linear algebraic group $G$ which is a quotient of $\pi_1^S(Y,y)$, a principal $G$-bundle $\pi:P\rightarrow Y$, such that $f$ is an $(X_y,\Delta_y)$-bundle associated to $\pi:P\rightarrow Y$.  
\end{thm}
\begin{proof}
	We pick by Theorem \ref{thm: existence of relative ample numerically flat bundle} an $f$-ample line bundle $\calL$ on $X$ such that $f_*(\calL^{\otimes m})$ are numerically flat for all $m> 0$. 
	In view of Theorem \ref{thm: main theorem in text}, it suffices to prove the following claim. 
	\begin{claim}
		Let $r$ be the Cartier index of an irreducible component $D$ of $\Delta$. 
		Then there exists an $m_0$ such that for any irreducible component $D$ of $\Delta$, $f_*(\calL^{\otimes m}(-rD))$ are numerically flat for all $m\geq m_0$. 
	\end{claim}
	To prove the claim, note that numerical flatness can be descended along base field extensions, so without loss of generality we assume $k$ is algebraically closed. 
	Let $y_0\in Y(k)$ be a point such that $(X_{y_0},\Delta_{y_0})$ is strongly $F$-regular. 
	By replacing $\calL$ by a multiple, we may assume that for every irreducible component $D$ of $\Delta$, 
	\begin{enumerate}
		\item $\calL(-rD)$ is $f$-ample , 
		\item the higher direct image $R^if_* (\calL^{\otimes m}(-rD))$ vanishes for all $i>0$ and $m>0$, and
		\item $S^0(X_{y_0},\Delta_{y_0}, \calL_{y_0}^{\otimes m}(-rD_{y_0}))=H^0(X_{y_0},\calL_{y_0}^{\otimes m}(-rD_{y_0}))$ for all $m>0$. 
	\end{enumerate}
	Let $C\rightarrow Y$ be any morphism from a smooth projective curve going through $y_0$, and denote $(Z,\Lambda):=(X,\Delta)\times_Y C,\ D':=D \times_Y C,\ \calL':= \calL|_Z$. 
	Then for all $i>0,m>0$ and all irreducible components $D$ of $\Delta$, we have $R^if'_* (\calL'^{\otimes m}(-rD'))=0$ by Corollary \ref{cor: comparison of higher direct images}. 
	
	Let $m_0$ be an integer such that $\frac{r_D}{m_0}\leq \mathrm{coeff}_\Delta (D)$ for every irreducible component $D$ of $\Delta$. 
	For any point $y_1\in Y(k)$, there exists always a curve $C$ going through $y_0$ and $y_1$. 
	For any $m\geq m_0$ and any $D$, all assumptions in Settings \ref{set: conditions for proving numerical flatness of ideal sheaves} and Proposition \ref{prop: numerical flatness of ideal sheaf} are satisfied, so the bundles $f'_*(\calL'^{\otimes m}(-rD'))$ are numerically flat. 
	So we can apply Theorem \ref{thm: main theorem in text} to $C$ and $\calL'^{\otimes m_0}$ and see that $f':(Z,\Lambda) \rightarrow C$ is isotrivial. 
	As $y_1$ is chosen arbitrarily, this shows also $f:(X,\Delta) \rightarrow Y$ is isotrivial. 
	Now every fibre of $f$ is strongly $F$-regular, and so we can apply \cite[Lemma 2.20 \& Proposition 2.23]{patakfalviSemipositivityPositiveCharacteristics2014} and deduce that after increasing $m_0$, the equality $S^0(X_y,\Delta_y, \calL_y^{\otimes m}(-rD_y))=H^0(X_y,\calL_y^{\otimes m}(-rD_y))$ holds for all $m\geq m_0$, all $D$ and all $y\in Y(k)$. 
	Hence we can apply Proposition \ref{prop: numerical flatness of ideal sheaf} to any curve $C \rightarrow Y$ (not necessarily going through $y_0$) again, and get that the bundles $f_*(\calL^{\otimes m}(-rD))|_C = f'_* \calL'^{\otimes m}(-rD')$ are numerically flat for all $m\geq m_0$, where the equality is due to Corollary \ref{cor: comparison of higher direct images}. 
	By \cite[Remark 5.2]{langerSfundamentalGroupScheme2011}, $f_*(\calL^{\otimes m}(-rD))$ is also numerically flat on the whole base $Y$. 
	This finishes the proof of the claim, and hence we may apply Theorem \ref{thm: main theorem in text} to $Y$ and $\calL^{\otimes m_0}$.
\end{proof}

\subsection{The weak Beauville--Bogomolov decomposition}
One special case that we are interested in is when we take the base $Y$ to be the Albanese of $X$. 
It is known that in the same assumption as Theorem \ref{thm: associated bundle relative anti-nef positive characteristic in text}, the Albanese map is a surjective fibration. 

\begin{prop}[{\cite[Theorem 1.1]{ejiriWhenAlbaneseMorphism2019}}]
	Let $(X,\Delta)$ be a normal projective pair and consider the Albanese $f:(X,\Delta) \rightarrow A$. 
	Assume that
	\begin{enumerate}
		\item the geometric generic fibre $(X_{\overline \eta}, \Delta_{\overline \eta})$ over the image of $f$ is sharply $F$-pure,
		\item $-K_X-\Delta$ is nef and of Cartier index coprime to $p$.
	\end{enumerate}
	Then the Albanese $(X,\Delta) \rightarrow A$ is a surjective fibration. 
\end{prop}

In general, it is hard to verify if the geometric generic fibre is sharply $F$-pure or strongly $F$-regular. 
But when the base is the Albanese, we have however the following sufficient condition:

\begin{prop}\label{prop: geometric generic fibre SFR over Albanese}
	Let $(X,\Delta)$ be a strongly $F$-regular and globally $F$-split normal projective pair, such that the Cartier index of $K_X+\Delta$ is coprime to $p$. 
	Let $f:(X,\Delta) \rightarrow A$ be the Albanese. 
	Then the following statements hold.
	\begin{enumerate}
		\item $X$ is Cohen-Macaulay,
		\item $f$ is a surjective fibration, and
		\item the geometric generic fibre $(X_{\overline \eta},\Delta_{\overline \eta})$ is strongly $F$-regular.
	\end{enumerate}
\end{prop}
\begin{proof}
	The statements can be descended along base changes to perfect fields, so without loss of generality we can assume that the base field $k$ is algebraically closed. 
	By \cite[Theorem 1.2 \& Theorem 1.3]{ejiriWhenAlbaneseMorphism2019}, the pair $(X,\Delta)$ is globally $F$-split over $A$, and $f$ is a surjective fibration. 
	Then by Lemma \ref{lem: base change of relative GFS}, the geometric generic fibre $(X_{\overline \eta},\Delta_{\overline \eta})$ is globally $F$-split over $\overline \eta$. 
	Since $k(\overline \eta)$ is $F$-finite, $\overline\eta$ is globally $F$-split over $\bbF_p$. 
	Hence $(X_{\overline \eta},\Delta_{\overline \eta})$ is globally $F$-split by Lemma \ref{lem: compositum of relative GFS}. 
	By \cite[Proposition 2.13]{patakfalviWeakBeauvilleBogomolovDecomposition2025}, 
	the relative global $F$-splitting of $X_{\overline \eta}$ over $\overline\eta$ implies $X_{\overline \eta}$ is regular in codimension one.
	Moreover, by \cite[Corollary 2.5]{hochsterTightClosureStrong1989}, the underlying variety $X$ is Cohen-Macaulay, and therefore $X_{\overline \eta}$ is Cohen-Macaulay too. 
	Using Serre's criterion, we conclude that $X_{\overline \eta}$ is normal.
	As $(X_{\overline \eta},\Delta_{\overline \eta})$ is normal and globally $F$-split, we can apply \cite[Theorem 6.3]{patakfalviWeakBeauvilleBogomolovDecomposition2025} and deduce that $(X_{\overline \eta},\Delta_{\overline \eta})$ is strongly $F$-regular. 
\end{proof}

We may then apply Theorem \ref{thm: associated bundle relative anti-nef positive characteristic in text} and deduce the weak Beauville--Bogomolov decomposition in positive characteristics.

\begin{cor}\label{cor: associated bundle abelian variety positive characteristic in text}
	Let $(X,\Delta)$ be a globally $F$-split strongly $F$-regular projective pair, such that $-K_X-\Delta$ is nef and of Cartier index coprime to $p$. 
	Let $(X,\Delta)\rightarrow A$ be the Albanese of $X$, and assume that $A$ has a rational point $0\in A(k)$. 
	Then $(X,\Delta)\rightarrow A$ is an $(X_0,\Delta_0)$-bundle associated to a principal $G$-bundle $P\rightarrow A$ of an algebraic group $G$, and $G$ is a quotient of $\pi_1^S(A,0)$.
\end{cor}

\subsection{Cases over complex numbers}
We briefly discuss the parallel consequences of Theorem \ref{thm: main theorem in text} over the complex numbers $\bbC$. 
The existence of a family of numerically flat bundles is shown in \cite{matsumuraStructureTheoremProjective2025}.

\begin{thm}[{\cite[Proof of Theorem 4.1]{matsumuraStructureTheoremProjective2025}}]\label{thm: existence of relative ample numerically flat bundle complex numbers}
	Let $f:(X,\Delta) \rightarrow Y$ be a surjective morphism from a projective klt pair to a smooth projective variety. 
	Assume the relative anti-log-canonical divisor $-K_{X/Y}-\Delta$ is nef. 
	Then there exists an $f$-ample line bundle $\calL$ such that $f_*(\calL^{\otimes m})$ are numerically flat for all $m\geq 0$.
\end{thm}

Hence by Theorem \ref{thm: main theorem in text}, we may deduce a result similar to Theorem \ref{thm: associated bundle relative anti-nef positive characteristic in text}. 

\begin{thm}\label{thm: associated bundle relative anti-nef complex numbers in text}
	Let $f:(X,\Delta) \rightarrow Y$ be a surjective fibration over the complex numbers $\bbC$ such that
	\begin{enumerate}
		\item $(X,\Delta)$ is projective klt,
		\item $Y$ is smooth projective, with a distinguished point $y\in Y$,
		\item $-K_{X/Y}-\Delta$ is nef.
	\end{enumerate}
	Then there exists a linear algebraic group $G$ which is a quotient of $\pi_1^S(Y,y)$, a principal $G$-bundle $\pi:P\rightarrow Y$, such that $f$ is an $(X_y,\Delta_y)$-bundle associated to $\pi:P\rightarrow Y$.  
\end{thm}
\begin{proof}
	We proceed similarly as in the proof of Theorem \ref{thm: associated bundle relative anti-nef positive characteristic in text}. 
	Pick a line bundle $\calL$ as in Theorem \ref{thm: existence of relative ample numerically flat bundle complex numbers}. 
	We just need to show the following claim:
	\begin{claim}
		For any irreducible component $D$ of $\Delta$, there exists an $m_0$ such that $f_*(\calL^{\otimes m}(-D))$ are numerically flat for all $m\geq m_0$. 
	\end{claim} 
	Indeed, the proof of Lemma \ref{lem: top intersection vanishes} is characteristic-free, and we can combine Lemma \ref{lem: top intersection vanishes} with \cite[Lemma A.4 \& Proposition A.11]{patakfalviWeakBeauvilleBogomolovDecomposition2025}, showing that the pullback bundle $f_*(\calL^{\otimes m}(-D))|_C$ is numerically flat for any morphism $C\rightarrow Y$ from a smooth curve $C$. 
	This in turn shows that $f_*(\calL^{\otimes m}(-D))$ is numerically flat, finishing the proof. 
\end{proof}

In particular, when $Y$ is the Albanese of $X$, we have the following Corollary. 

\begin{cor}\label{cor: associated bundle abelian variety complex numbers in text}
	Let $(X,\Delta)$ be a projective klt pair such that $-K_X-\Delta$ is nef. 
	Then the Albanese $(X,\Delta)\rightarrow A$ is an $(X_0,\Delta_0)$-bundle associated to a principal $G$-bundle $P\rightarrow A$ of an algebraic group $G$, and $G$ is a quotient of $\pi_1^S(A,0)$. 
\end{cor}

Since there exists a canonical map $\pi_1(A,0) \rightarrow \pi_1^S(A,0)$, we directly recover \cite[Corollary 4.2]{matsumuraStructureTheoremProjective2025}. 

\section{Bertini for globally $F$-split strongly $F$-regular boundary}\label{section: Bertini}
In this section, we prove Theorem \ref{thm: Bertini of GFS+SFR} and deduce Theorem \ref{thm: associated bundle semi-ample} as a consequence.
Let us explain shortly the ideas in the proof. 
We consider a strongly $F$-regular pair $(X,\Delta)$, a base point free linear system $L$ on $X$, and the universal effective divisor $\widetilde\Gamma$ on $(X,\Delta)\times \bbP(L)$. 
We first perform a careful evaluation to show that if $(X, \Delta+ \frac{1}{m}\widetilde\Gamma|_t)$ is globally $F$-split at a closed point $t\in \bbP(L)$, then so is $(X\times U, \Delta\times U + \frac{1}{m}\widetilde\Gamma|_U)$ on an open $t\in U \subset \bbP(L)$. 
This shows the general global $F$-splitting. 
Next, according to \cite[Theorem 4.19]{tanakaBertiniTheoremsAdmitting2024}, the fibre $(X_\eta, \Delta_\eta + \frac{1}{m} \widetilde \Gamma|_\eta)$ at the generic point $\eta\in \bbP(L)$ is strongly $F$-regular. 
However in general, the strong $F$-regularity is only an open condition around a perfect point on the base (\cite[Theorem B]{patakfalviFsingularitiesFamilies2018}). 
Here, a perfect point means a morphism from $\Spec K$ to the base where $K$ is a perfect field. 
To bridge ourselves to the strong $F$-regularity of $(X_{\overline\eta}, \Delta_{\overline\eta} + \frac{1}{m}\widetilde\Gamma|_{\overline\eta})$, we note that the global $F$-splitting of $(X_\eta, \Delta_\eta + \frac{1}{m} \widetilde \Gamma|_\eta)$ implies the global $F$-splitting of $(X_{\overline\eta}, \Delta_{\overline\eta} + \frac{1}{m}\widetilde\Gamma|_{\overline\eta})$ by a generalization of \cite[Corollary 2.5]{gongyoRationalConnectednessGlobally2015} to the case of pairs, so the geometric generic strong $F$-regularity follows from \cite[Theorem 6.3]{patakfalviWeakBeauvilleBogomolovDecomposition2025}. 
Finally we manage to descend the strong $F$-regularity from a perfect base field extension to the original $F$-finite base field, finishing the proof. 

Throughout this section, we apply the following settings: 

\begin{set}
	We work over an $F$-finite field $k$ of characteristic $p>0$. 
	For a pair $(X,\Delta)$, we assume 
	\begin{enumerate}
		\item $X$ is normal and proper over $k$,
		\item $l:=H^0(X,\calO_X)$ is a separable field extension over $k$, 
		\item the log-canonical divisor $K_X+\Delta$ has a Cartier index coprime to $p$.
	\end{enumerate}
\end{set}

\subsection{Bertini-type theorem}
First we show that the global $F$-splitting is stable under $F$-finite base changes. 
Indeed, the proof in the boundary-free case (see \cite[Lemma 2.4 \& Corollary 2.5]{gongyoRationalConnectednessGlobally2015}) still works in the settings with a boundary. 
We need, however, argue first that the divisors on base changes are indeed well-defined. 

\begin{lem}\label{lem: GFS implies geometrically normal}
	Under the settings above, if $(X,\Delta)$ is globally $F$-split, then $X$ is geometrically normal. 
\end{lem}
\begin{proof}
	By \cite[Tag 038O \& Tag 033G]{stacks-project}, it suffices to show that $Y:= X\times_k k^{\frac{1}{p^e}}$ is normal for all $e>0$. 
	To start, we remark that $(X,0)$ is globally $F$-split too. 
	As affine opens of globally $F$-split schemes are still globally $F$-split, we can assume $X=\Spec A, Y= \Spec A\otimes_k k^{\frac{1}{p^e}}$. 
	For simplicity denote $B:= A\otimes_k k^{\frac{1}{p^e}}$.
	Let $H: B \rightarrow A^{\frac{1}{p^e}}$ be the relative Frobenius of $A$ over $k$.
	By the proof of \cite[Corollary 2.5]{gongyoRationalConnectednessGlobally2015}, $X$ is globally $F$-split relative to $k$, that is, the map $H: B \rightarrow A^{\frac{1}{p^e}}$ has a $B$-linear section $\sigma$. 
	Let $z=\frac{b}{c}$ be an element in $\mathrm{Frac}(B)$ which is integral over $B$. 
	As $A$ is normal, we have $z\in A^{\frac{1}{p^e}}$. 
	Then one can apply $\sigma$ on the equation $cz=b$ and get $\sigma(z)=z$, hence $z\in B$.
\end{proof}

\begin{prop}\label{prop: GFS is stable under base change}
	Under the settings above, if the pair $(X,\Delta)$ is globally $F$-split, then so is it relatively over $k$. 
	In particular, for any field extension $k\subset K$ where $K$ is $F$-finite, the pair $(X_K, \Delta_K)$ after base change is globally $F$-split as well.
\end{prop}
\begin{proof}
	Let $e$ be an integer such that the map $\calO_X\rightarrow F^e_*\calO_X(\lceil(p^e-1)\Delta\rceil)$ is split. 
	We show first that $(X,\Delta)$ is globally $F$-split over $k$.  
	Let $(Y,\Theta)$ be the base change $(X,\Delta)\times_k k^{\frac{1}{p^e}}$. 
	We name some morphisms as in the following diagram:
	\[\begin{tikzcd}
		X &&&&& \\
		& Y && X \\
		&&&& Y & X \\
		&& {k^{\frac{1}{p^{2e}}}} && {k^{\frac{1}{p^e}}} & k.
		\arrow["H"{description}, from=1-1, to=2-2]
		\arrow["{F^e}", from=1-1, to=2-4]
		\arrow[from=1-1, to=4-3]
		\arrow["G"{description}, from=2-2, to=2-4]
		\arrow["{F^e}"', from=2-2, to=3-5]
		\arrow[from=2-2, to=4-3]
		\arrow["H"{description}, from=2-4, to=3-5]
		\arrow["{F^e}", from=2-4, to=3-6]
		\arrow["f"', from=2-4, to=4-5]
		\arrow["G"{description}, from=3-5, to=3-6]
		\arrow[from=3-5, to=4-5]
		\arrow["\lrcorner"{anchor=center, pos=0.125}, draw=none, from=3-5, to=4-6]
		\arrow["f", from=3-6, to=4-6]
		\arrow["{F^e}"', from=4-3, to=4-5]
		\arrow["{F^e}"', from=4-5, to=4-6]
	\end{tikzcd}\]
	Consider the following composition of $\calO_X$-module morphisms:
	\[
		\calO_X \rightarrow G_*\calO_Y \rightarrow G_*F^e_*\calO_Y(\lceil (p^e-1)\Theta \rceil) \rightarrow F^{2e}_*\calO_X(\lceil (p^{2e}-1)\Delta \rceil).
	\]
	Taking dual we get
	\begin{align*}
		\SheafHom_{\calO_X}(F^{2e}_*\calO_X(\lceil (p^{2e}-1)\Delta \rceil),\calO_X) &\rightarrow 
		\SheafHom_{\calO_X}(G_*F^e_*\calO_Y(\lceil (p^e-1)\Theta \rceil), \calO_X)\\ \rightarrow
		\SheafHom_{\calO_X}(G_*\calO_Y, \calO_X)&\rightarrow \calO_X.
	\end{align*}
	We now compute the two middle terms. 
	First we have the following identification by Grothendieck duality and Lemma \ref{lem: relative dualizing sheaf of field extension}:
	\begin{align*}
		\SheafHom_{\calO_X}(G_*\calO_Y,\calO_X) 
		&\cong G_*G^!\calO_X = \cong G_*\calO_Y,\\
		\SheafHom_{\calO_X}(G_*F^e_*\calO_Y(\lceil (p^e-1)\Theta\rceil), \calO_X)
		&=G_*\SheafHom_{\calO_Y}(F^e_*\calO_Y(\lceil (p^e-1)\Theta \rceil), \calO_Y).
	\end{align*}
	Taking global sections of the dual chain we get
	\begin{align*}
		\Hom_{\calO_X}(F^{2e}_*\calO_X(\lceil (p^{2e}-1)\Delta \rceil),\calO_X) \rightarrow
		\Hom_{\calO_Y}(F^e_*\calO_Y(\lceil (p^e-1)\Theta \rceil), \calO_Y) \overset{\beta}{\rightarrow}
		l\otimes_k k^{\frac{1}{p^e}}\rightarrow 
		l.
	\end{align*}
	By the assumption of global $F$-splitting and Lemma \ref{lem: multiplies of GFS index are still GFS}, the whole composition is a surjective map of $l$-vector spaces, which implies $\beta$ is non-zero. 
	As $l/k$ is separable, the ring $l\otimes_k k^{\frac{1}{p^e}}$ is indeed a field. 
	Now $\beta$ is a non-zero map of $l\otimes_k k^{\frac{1}{p^e}}$-vector spaces with a one-dimensional codomain, hence must be surjective, which in turn shows that $(Y,\Theta)$ is globally $F$-split. 
	The map $\calO_Y \rightarrow H_*\calO_X(\lceil (p^e-1)\Delta \rceil)$ is split as the split map $\calO_Y \rightarrow F^e_*\calO_Y(\lceil (p^e-1)\Theta \rceil)$ factors through it, and therefore $(X,\Delta)$ is globally $F$-split over $k$. 
	
	For the general global $F$-splitting of $(X_K,\Delta_K)$, we may apply first Lemma \ref{lem: base change of relative GFS} to see that $(X_K,\Delta_K)$ is globally $F$-split over $K$. 
	Now $\Spec K$ is globally $F$-split as $K$ is $F$-finite, and we deduce that $(X_K,\Delta_K)$ is absolutely globally $F$-split by Lemma \ref{lem: compositum of relative GFS}. 
\end{proof}

We then give a Bertini type theorem on the boundary of a globally $F$-split pair.

\begin{prop}\label{prop: relative GFS on an open}
	Assume the pair $(X,\Delta+\Delta')$ is globally $F$-split, where we separate the boundary into two parts such that
	\begin{enumerate}
		\item both $\Delta$ and $\Delta'$ are effective $\bbQ$-Cartier $\bbQ$-divisors, and
		\item the Cartier index of $\Delta'$ is coprime to $p$.
	\end{enumerate}
	Pick integer $m$ coprime to $p$ such that $\Gamma=m\Delta'$ is an integral Cartier divisor, and consider the complete linear system $|\Gamma|$. 
	Denote with $\widetilde \Gamma$ the universal hyperplane on $X\times \bbP(|\Gamma|)$. 
	Then there is an open $U\subset \bbP(|\Gamma|)$ containing the closed point $t\in U(l)$ corresponding to $\Gamma$, such that 
	\begin{enumerate}
		\item $(X\times U,\pr_1^*\Delta+\frac{1}{m}\widetilde \Gamma|_U)$ is globally $F$-split over $U$,
		\item $(X\times U,\pr_1^*\Delta+\frac{1}{m}\widetilde \Gamma|_U)$ is globally $F$-split over $l$, 
		\item $(X\times U,\pr_1^*\Delta+\frac{1}{m}\widetilde \Gamma|_U)$ is globally $F$-split. 
	\end{enumerate}
\end{prop}
\begin{proof}
	Take an integer $e$ such that $m$ divides $p^e-1$. 
	Let $(Y,\Theta+\Theta'):=(X,\Delta+\Delta')\times_l l^{\frac{1}{p^e}}$ be the Frobenius base change, and let $\Pi:=m\Theta'$, which is the base change of $\Gamma$. 
	Similarly, we denote with $\widetilde\Pi$ the universal divisor on $Y\times \bbP(|\Pi|)$.
	By Lemma \ref{lem: multiplies of GFS index are still GFS}, we may replace $e$ by a multiple and assume the map $\calO_X \rightarrow F_*^e \calO_X(\lceil (p^e-1)(\Delta+\Delta') \rceil)$ is split. 
	By Proposition \ref{prop: GFS is stable under base change}, the map $\calO_Y \rightarrow F_{X/l,*}^e \calO_X(\lceil (p^e-1)(\Delta+\Delta') \rceil)$ is also split. 
	Consider the following diagram:
	\[\begin{tikzcd}
		{X\times \bbP(|\Gamma|)} &&&& \\
		& {Y\times \bbP(|\Gamma|)} && {Y\times \bbP(|\Pi|)} & {X\times \bbP(|\Gamma|)} \\
		& {\bbP(|\Gamma|)} && {\bbP(|\Pi|)} & {\bbP(|\Gamma|).} \\
		&&& {\Spec l^{\frac{1}{p^e}}} & {\Spec l}
		\arrow["{{F^e_{X/l}\times \Id}}"{description}, from=1-1, to=2-2]
		\arrow["{{F^e}}", from=1-1, to=2-5]
		\arrow["{{{\pr_2}}}"', from=1-1, to=3-2]
		\arrow["{{\Id \times F^e_{\bbP(\Gamma)/l}}}"', from=2-2, to=2-4]
		\arrow[from=2-2, to=3-2]
		\arrow["\lrcorner"{anchor=center, pos=0.125}, draw=none, from=2-2, to=3-4]
		\arrow[from=2-4, to=2-5]
		\arrow["{{{\pr_2}}}", from=2-4, to=3-4]
		\arrow["\lrcorner"{anchor=center, pos=0.125}, draw=none, from=2-4, to=3-5]
		\arrow["{{{\pr_2}}}", from=2-5, to=3-5]
		\arrow["{{F^e_{\bbP(|\Gamma|)/l}}}"', from=3-2, to=3-4]
		\arrow[from=3-2, to=4-4]
		\arrow[from=3-4, to=3-5]
		\arrow[from=3-4, to=4-4]
		\arrow["\lrcorner"{anchor=center, pos=0.125}, draw=none, from=3-4, to=4-5]
		\arrow[from=3-5, to=4-5]
		\arrow["{{F^e}}"', from=4-4, to=4-5]
	\end{tikzcd}\]
	Applying Lemma \ref{lem: base change of relative GFS} to the above mentioned split map, the map 
	\[
		\calO_{Y\times \bbP(|\Gamma|)} \rightarrow \left( F_{X/l}^e\times\Id \right)_* \calO_{X\times \bbP(|\Gamma|)}\left(\left\lceil (p^e-1)\left( \pr_1^*\Delta+\pr_1^*\Delta' \right) \right\rceil\right)
	\] 
	is also split.
	Let $V$ be an affine open in $\bbP(|\Gamma|)$ containing $t$. 
	By Lemma \ref{lem: duality}, we get a correspondence 
	\begin{gather*}
		\Hom_{\calO_{Y\times V} \text{-mod}} \left( \left(F_{X/l}^e\times\Id \right)_* \calO_{X\times V}\left(\left\lceil (p^e-1) \left( \pr_1^*\Delta+\frac{1}{m}\widetilde\Gamma \right) \right\rceil\right) , \calO_{Y\times V} \right)\\
		\parallel\\
		H^0\left( X\times V, \calO_{X\times V}\left(\left\lfloor (1-p^e)\left( \pr_1^*(K_{X}+\Delta)+\frac{1}{m}\widetilde\Gamma \right) \right\rfloor\right)\right)\\
		\parallel\\
		H^0\left( X\times V, \pr_1^*\calO_{X}\left(\left\lfloor (1-p^e)\left( K_{X}+\Delta+\frac{1}{m}\Gamma \right) \right\rfloor\right)\right), 
	\end{gather*}
	where the latter equality is because $\calO_{X\times\bbP(|\Gamma|)}(\widetilde\Gamma)=\pr_1^*\calO_X(\Gamma)\otimes\pr_2^*\calO_{\bbP(|\Gamma|)}(1)$, and $\calO_{\bbP(|\Gamma|)}(1)|_V$ is trivial. 
	There is then a canonical section on $H^0(X\times V, \pr_1^*\calO_X(\lfloor (1-p^e)(K_X+\Delta+\frac{1}{m}\Gamma) \rfloor))$ defined by the pullback of the section of $H^0(X,\calO_X(\lfloor (1-p^e)(K_X+\Delta+\frac{1}{m}\Gamma) \rfloor))$ that gives the global $F$-splitting on $X$. 
	The section defines then a homomorphism $(F_{X/l}^e\times\Id )_* \calO_{X\times V}(\lceil (p^e-1) ( \pr_1^*\Delta+\frac{1}{m}\widetilde\Gamma ) \rceil) \rightarrow \calO_{Y\times V}$, and we take its precomposition with the natural map $\calO_{Y\times V} \rightarrow (F_{X/l}^e\times\Id )_* \calO_{X\times V}(\lceil (p^e-1)(\pr_1^*\Delta+\frac{1}{m}\widetilde\Gamma) \rceil)$. 
	The composition is nothing but a section in $H^0(Y\times V, \calO_{Y\times V})$, whose restriction to $Y_t$ is non-zero, hence itself is invertible on an open $U\subset V$ containing $t$.
	This shows that the map $\calO_{Y\times U} \rightarrow (F_{X/l}^e\times\Id )_*\calO_{X\times U}(\lceil (p^e-1)(\pr_1^*\Delta+\frac{1}{m}\widetilde\Gamma) \rceil)$ is split, that is, $(X\times U,\pr_1^*\Delta+\frac{1}{m}\widetilde \Gamma|_U)$ is globally $F$-split over $U$. 
	For the global $F$-splitting of $(X\times U,\pr_1^*\Delta+\frac{1}{m}\widetilde \Gamma|_U)$ over $l$ and the absolute $F$-splitting, it suffices to notice that $U$ is globally $F$-split over $\Spec l$ and $\Spec l$ is globally $F$-split over $\bbF_p$, which allows us to apply Lemma \ref{lem: compositum of relative GFS}. 
\end{proof}

\begin{cor}\label{cor: the generic fibre is GFS}
	Succeeding the notations above. 
	Let $\eta$ be the generic point of $\bbP(|\Gamma|)$. Then $(X_\eta, \Delta_\eta+\frac{1}{m}\widetilde\Gamma|_\eta)$ is globally $F$-split over $\eta$ and absolutely globally $F$-split. 
\end{cor}
\begin{proof}
	Since $(X\times U,\pr_1^*\Delta+\frac{1}{m}\widetilde \Gamma|_U)$ is globally $F$-split over $U$, Lemma \ref{lem: base change of relative GFS} tells us that $(X_\eta, \Delta_\eta+\frac{1}{m}\widetilde\Gamma|_\eta)$ is globally $F$-split over $\eta$. 
	Moreover, $k(\eta)$ is a purely transcendental, finitely generated extension over $l$, hence is $F$-finite and globally $F$-split. 
	Then we can apply Lemma \ref{lem: compositum of relative GFS} and get the absolute global $F$-splitting of $(X_\eta, \Delta_\eta+\frac{1}{m}\widetilde\Gamma|_\eta)$. 
\end{proof}

The proof of the next Corollary is identical to the argument above, by replacing $\eta$ with a closed point $\Spec l \rightarrow U$. 

\begin{cor}\label{cor: Bertini for GFS}
	Succeeding the notations above. For a general member $\Gamma' \in |\Gamma|$, the pair $(X,\Delta+\frac{1}{m}\Gamma')$ is globally $F$-split. 
\end{cor}

Now we turn to study the base change properties of strong $F$-regularity. 
It is well known that they are not stable under field extensions. 
However recall from \cite[Theorem 6.3]{patakfalviWeakBeauvilleBogomolovDecomposition2025} that geometric sharp $F$-purity ensures stable base change of strong $F$-regularity. 
In their settings the base field is assumed to be perfect, but the proof works essentially still in $F$-finite settings. 

\begin{prop}[{\cite[Theorem 6.3]{patakfalviWeakBeauvilleBogomolovDecomposition2025}}]\label{prop: SFR stable under base change if GFS}
	Assume $(X,\Delta)$ is strongly $F$-regular and geometrically sharply $F$-pure, that is, $(X,\Delta)_{\bar k}$ is sharply $F$-pure. 
	Then $(X,\Delta)_{\bar k}$ is also strongly $F$-regular. 
\end{prop}

The situation for the converse direction is better. 
We may show that strong $F$-regularity descends from perfect base fields. 

\begin{prop}\label{prop: SFR descends along perfect field extension}
	Let $k\subset k'$ be a field extension such that $k'$ is perfect. 
	If $(X,\Delta)_{k'}$ is strongly $F$-regular, then so is $(X,\Delta)$. 
\end{prop}
\begin{proof}
	Let $\tau(X,\Delta)$ be the test ideal (\cite[Section 2.2]{schwedeCentersFpurity2010}) associated to the pair $(X,\Delta)$. 
	By \cite[Proposition 6.18]{schwedeBehaviorTestIdeals2012}, $(X,\Delta)$ is strongly $F$-regular if and only if $\tau(X,\Delta)=\calO_X$. 
	By \cite[Theorem A]{patakfalviFsingularitiesFamilies2018}, where we let $V=k$, there is a finite extension $k''=k^{\frac{1}{p^e}}$ over $k$ such that $\tau(X_{k''},\Delta_{k''})=\calO_{X_{k''}}$, hence $(X,\Delta)_{k''}$ is strongly $F$-regular. 
	Let $\pi: X_{k''} \rightarrow X$ be the natural map. 
	We pick then a $k$-linear map $k''\rightarrow k$ that splits the inclusion $k\rightarrow k''$. 
	Pulling back to $X$, we get a surjective $\calO_X$-module homomorphism $\pi_*\calO_{X_{k''}} \rightarrow \calO_X$ that splits $\calO_X \rightarrow \pi_*\calO_{X_{k''}}$. 
	We denote this homomorphism with $\frakT$, following the notations in \cite{schwedeBehaviorTestIdeals2012}. 
	The map $\frakT$ corresponds to an effective divisor $R_\frakT$ linearly equivalent to $K_{X_{k''}}-\pi^*K_X$ by Lemma \ref{lem: duality}. 
	Since $K_{X_{k''}}=\pi^*K_X$, the divisor $R_\frakT$ has to be trivial. 
	So by \cite[Theorem 6.25]{schwedeBehaviorTestIdeals2012}, we have $\tau(X,\Delta)=\frakT(\tau(X_{k''},\Delta_{X_{k''}}))=\frakT(\calO_{X_{k''}})=\calO_X$. 
\end{proof}

We are now ready to prove Theorem \ref{thm: Bertini of GFS+SFR}. 
For convenience to readers we state the theorem here again. 

\begin{thm}\label{thm: Bertini of GFS+SFR in text}
	Let $(X,\Delta)$ be a strongly $F$-regular proper pair over an $F$-finite field $k$ of characteristic $p>0$, such that
	\begin{enumerate}
		\item $H^0(X,\calO_X)$ is separable over $k$, and
		\item $K_X+\Delta$ has Cartier index coprime to $p$.
	\end{enumerate}
	Let $\Delta'$ be an effective $\bbQ$-Cartier $\bbQ$-divisor such that $(X,\Delta+\Delta')$ is globally $F$-split, and $m>1$ an integer coprime to $p$ such that $m\Delta'$ is Cartier and $|m\Delta'|$ is base point free. 
	Then for a general member $\Gamma\in |m\Delta'|$, the pair $(X,\Delta+\frac{1}{m}\Gamma)$ is globally $F$-split and strongly $F$-regular.
\end{thm}

\begin{proof}
	Let $\bbP(|m\Delta'|)$ be the total space of effective divisors linearly equivalent to $m\Delta'$, and consider the universal divisor $\widetilde\Gamma$ on $X\times \bbP(|m\Delta'|)$. 
	Let $\eta\in\bbP(|m\Delta'|)$ be the generic point. 
	Note that $k(\eta)$ is a purely transcendental and finitely generated extension of $l$, and hence is in particular $F$-finite. 
	The global $F$-splitting of a general $\Gamma\in|m\Delta'|$ follows from Corollary \ref{cor: Bertini for GFS}. 
	We proceed then to prove its strong $F$-regularity.
	By \cite[Theorem 4.19]{tanakaBertiniTheoremsAdmitting2024}, the pair $(X_\eta,\Delta_\eta+\widetilde\Gamma|_\eta)$ is purely $F$-regular. 
	Then by Corollary \ref{cor: midpoint of SFR and PFR is SFR}, the pair $(X_\eta,\Delta_\eta+\frac{1}{m}\widetilde\Gamma|_\eta)$ is strongly $F$-regular. 
	It is also globally $F$-split by Corollary \ref{cor: the generic fibre is GFS}. 
	Then we apply Proposition \ref{prop: GFS is stable under base change} and deduce that the geometric generic fibre $(X_{\bar\eta},\Delta_{\bar\eta}+\frac{1}{m}\widetilde\Gamma|_{\bar\eta})$ is also globally $F$-split. 
	By Proposition \ref{prop: SFR stable under base change if GFS}, the geometric generic fibre is also strongly $F$-regular. 
	Finally, we can apply \cite[Theorem B]{patakfalviFsingularitiesFamilies2018} to see that for a general $\Gamma\in|m\Delta'|$, the pair $(X,\Delta+\frac{1}{m}\Gamma)_{\bar l}$ is strongly $F$-regular, and apply Proposition \ref{prop: SFR descends along perfect field extension} to see that $(X,\Delta+\frac{1}{m}\Gamma)$ is strongly $F$-regular. 
\end{proof}

\subsection{Finite decomposition when $-K_X-\Delta$ is semi-ample or $k$ is finite}
Applying Theorem \ref{thm: Bertini of GFS+SFR in text} to a strongly $F$-regular globally $F$-split pair $(X,\Delta)$ such that $-K_X-\Delta$ is semi-ample, we can pick a complement $\Delta' \sim_\bbQ -K_X-\Delta$ such that $(X,\Delta + \Delta')$ is still strongly $F$-regular globally $F$-split, hence

\begin{cor}\label{cor: associated bundle semi-ample in text}
	Let $(X,\Delta)$ be a globally $F$-split strongly $F$-regular projective pair over an infinite perfect field $k$ of characteristic $p>0$, such that $m(-K_X-\Delta)$ is Cartier and base point free for some integer $m$ coprime to $p$. 
	Let $(X,\Delta)\rightarrow A$ be the Albanese of $X$, and assume that $A$ has a rational point $0\in A(k)$. 
	Then $(X,\Delta)\rightarrow A$ is an $(X_0,\Delta_0)$-bundle associated to an isogeny $A'\rightarrow A$. 
\end{cor}
\begin{proof}
	By Corollary \ref{cor: GFS is log CY}, there exists $\Delta'\sim_\bbQ -K_X-\Delta$ such that
	\begin{enumerate}
		\item $(X,\Delta+\Delta')$ is globally $F$-split, and
		\item $\Delta'$ has Cartier index coprime to $p$. 
	\end{enumerate}
	Let $m'$ be an integer coprime to $p$, such that $\Gamma:=m'\Delta'$ is Cartier and $|\Gamma|$ is base point free.
	Then we may apply Theorem \ref{thm: Bertini of GFS+SFR in text} and replace $\Gamma$ by a general member in the complete linear system. 
	We need the infinitude assumption on $k$ here, as otherwise an open set in $\bbP(|\Gamma|)$ might contain no rational point. 
	The new resulting $\Delta'=\frac{1}{m'}\Gamma$ still satisfies the two above mentioned properties, and we additionally obtain that $(X,\Delta+\Delta')$ is strongly $F$-regular. 
	Without loss of generality we may assume furthermore that $\frac{1}{m'}$ is not equal to any other coefficient in $\Delta$, and $\Gamma$ does not agree with any irreducible component of $\Delta$. 
	
	By \cite[Theorem 9.3]{patakfalviWeakBeauvilleBogomolovDecomposition2025}, the natural map $I=\mathrm{Isom}_A((X,\Delta+\Delta',\calL),A\times (X_0,\Delta_0+\Delta_0',\calL_0)) \rightarrow A$ is surjective for some relative polarization $\calL$, hence $I\rightarrow A$ is an $\Aut_{(X_0,\Delta_0+\Delta'_0,\calL_0)}$-torsor. 
	Moreover, there is a canonical isomorphism $I\times_A (X,\Delta+\Delta') \cong I\times (X_0,\Delta_0+\Delta'_0)$. 
	We claim that the $\Aut_{(X_0,\Delta_0+\Delta'_0,\calL_0)}$-action on $I\times (X_0,\Delta_0+\Delta'_0)$ is diagonal. 
	Indeed, consider the evaluation map
	\begin{align*}
		\mathrm{ev}: I\times (X_0,\Delta_0+\Delta'_0) &\rightarrow (X,\Delta+\Delta'),\\
						(\varphi,a) &\mapsto \varphi^{-1}(a).
	\end{align*}
	The map is clearly invariant under the diagonal action $\sigma\cdot(\varphi,a):=(\sigma(\varphi),\sigma(a))$ by $\Aut_{(X_0,\Delta_0+\Delta'_0,\calL_0)}$, and hence descends to a morphism $(I\times (X_0,\Delta_0+\Delta'_0))/G \rightarrow X$, which is an isomorphism because it is so after pulling back along $I\rightarrow A$.  
	Therefore, the Albanese $(X,\Delta+\Delta')\rightarrow A$ is an $(X_0,\Delta_0+\Delta'_0)$-bundle associated to $I\rightarrow A$. 
	By \cite[Proposition 10.1]{patakfalviWeakBeauvilleBogomolovDecomposition2025}, the torsor $I\rightarrow A$ is a finite map, and is dominated by an isogeny $A'\rightarrow A$ by \cite[Proposition 11.4]{patakfalviWeakBeauvilleBogomolovDecomposition2025}. 
	Hence $(X,\Delta+\Delta')\rightarrow A$ is also an $(X_0,\Delta_0+\Delta'_0)$-bundle associated to $A' \rightarrow A$. 
	As $\Supp(\Delta')$ does not agree with any irreducible component of $\Delta$, and the coefficient is not equal to any coefficient in $\Delta$, the action of $\Aut_{(X_0,\Delta_0+\Delta'_0,\calL_0)}$ must send $\Delta'$ to itself. 
	So we can remove now the boundary $\Delta'$, and $(X,\Delta)\rightarrow A$ is an $(X_0,\Delta_0)$-bundle associated to $A' \rightarrow A$. 
\end{proof}

Ejiri posed in \cite{ejiriSplittingAlgebraicFiber2023a} the question for a general base.

\begin{que}[{\cite[Question 1.16]{ejiriSplittingAlgebraicFiber2023a}}]
	Let $f:(X,\Delta)\rightarrow Y$ be a surjective morphism such that $-K_{X/Y}-\Delta$ is semi-ample. 
	Let $y\in Y(k)$ be a rational point. 
	Does there exist a finite cover $Y'\rightarrow Y$ that trivializes $f$, i.e. $Y'\times_Y (X,\Delta) \cong Y'\times (X_y,\Delta_y)$?
\end{que}

Nori introduced in \cite{noriFundamentalGroupscheme1982} the notion of an essentially finite vector bundle. 
Recall that a vector bundle $\calE$ is called \textbf{finite} if there exist two polynomials $u\neq v$ such that $u(\calE)\cong v(\calE)$, where sum and product are computed by direct sum and tensor product. 
A vector bundle is called \textbf{essentially finite} if it is a subquotient of finite vector bundles. 
Biswas and Dos Santos proved in \cite[Theorem 2]{biswasVectorBundlesTrivialized2011} that a vector bundle $\calE$ on $Y$ is essentially finite if and only if there exists a finite surjective morphism $f:Z\rightarrow Y$ such that $f^*\calE$ is trivial. 
If $\calE$ is essentially finite, then the group scheme Tannakian dual to $\langle \calE \rangle^\otimes$ is finite, see \cite[Proposition 3.10]{noriFundamentalGroupscheme1982}. 
So by Theorem \ref{thm: main theorem in text}, the question can be reduced to the following (probably harder) question:

\begin{que}
	Let $f:(X,\Delta)\rightarrow Y$ be a surjective morphism such that $-K_{X/Y}-\Delta$ is semi-ample. 
	Does there exist an $f$-ample line bundle $\calL$ on $X$ such that $f_*(\calL^{\otimes n})$ and $f_*(\calL^{\otimes n}(-D))$ are essentially finite for all $n$ and all irreducible components $D$ of $\Delta$?
\end{que}

\begin{que}
	Let $f:(X,\Delta)\rightarrow Y$ be a surjective morphism such that $-K_{X/Y}-\Delta$ is semi-ample. 
	Does there exist an $f$-ample line bundle $\calL$ on $X$ such that
	\begin{enumerate}
		\item $f_*\calL$ is essentially finite,
		\item $f_*(\calL^{\otimes n})$ and $f_*(\calL^{\otimes n}(-D))$ are numerically flat for all $n$ and all irreducible components $D$ of $\Delta$?
	\end{enumerate}  
\end{que}

If we restrict ourselves to work over a finite field, then indeed every numerically flat bundle is essentially finite. 

\begin{prop}\label{prop: numerically flat bundle over finite field}
	Let $Y$ be a smooth projective variety over a finite field $k$. 
	Then every numerically flat bundle on $Y$ is essentially finite. 
\end{prop}
\begin{proof}
	Simply combine \cite[Lemma 2.5]{patakfalviWeakBeauvilleBogomolovDecomposition2025} and \cite[Theorem 2]{biswasVectorBundlesTrivialized2011}. 
\end{proof}

Therefore, we obtain a finiteness result for the weak Beauville--Bogomolov decomposition when the base field is finite. 

\begin{thm}\label{thm: associated bundle finite field in text}
	Let $(X,\Delta)$ be a globally $F$-split strongly $F$-regular projective pair over a finite field $k$ of characteristic $p>0$, such that $-K_X-\Delta$ is nef and of Cartier index coprime to $p$. 
	Then the Albanese $(X,\Delta)\rightarrow A$ is an $(X_0,\Delta_0)$-bundle associated to an isogeny $A'\rightarrow A$, where $0\in A(k)$ is a rational point. 
\end{thm}
\begin{proof}
	By \cite[Theorem 2]{langAlgebraicGroupsFinite1956}, the Albanese $A$ has a rational point $0\in A(k)$, so we may apply Corollary \ref{cor: associated bundle abelian variety positive characteristic in text} and deduce that the Albanese $(X,\Delta)\rightarrow A$ is an $(X_0,\Delta_0)$-bundle associated to a principal $G$-bundle $P\rightarrow A$, where $G$ is the algebraic group Tannakian dual to $\langle f_*\calL \rangle^\otimes$. 
	By Proposition \ref{prop: numerically flat bundle over finite field}, the group $G$ is finite. 
	So $P\rightarrow A$ is dominated by an isogeny $A'\rightarrow A$ by \cite[Proposition 11.4]{patakfalviWeakBeauvilleBogomolovDecomposition2025}, and $(X,\Delta)\rightarrow A$ is an $(X_0,\Delta_0)$-bundle associated to $A'\rightarrow A$. 
\end{proof}

\section{Homogeneous fibrations over an abelian variety}\label{section: homogeneous fibration}
Recall Definition \ref{defn: homogeneous fibration} of a homogeneous morphism. 
In this section, we prove Theorem \ref{thm: associated bundles are homogeneous}. 
For convenience to readers we restate the theorem. We work over an algebraically closed field $k$ of characteristic $p\geq 0$.

\begin{thm}\label{thm: associated bundles are homogeneous in text}
	Let $f:(X,\Delta) \rightarrow A$ be a fibre bundle associated to the S-universal cover $\widetilde A^S \rightarrow A$, where $A$ is an abelian variety. 
	Then $f$ is homogeneous.
\end{thm}

\begin{proof}
	By \cite{miyanishiSomeRemarks1973} and \cite{brionHomogeneousVectorBundles2020}, the S-universal cover $\widetilde A^S$ can be endowed with a commutative group scheme structure such that the following sequence is exact:
	\[\begin{tikzcd}
		0 & {\pi_1^S(A,0)} & {\widetilde A^S} & A & 0.
		\arrow[from=1-1, to=1-2]
		\arrow[from=1-2, to=1-3]
		\arrow[from=1-3, to=1-4]
		\arrow[from=1-4, to=1-5]
	\end{tikzcd}\]
	Let $a\in A(k)$ be any rational point and let $t_a$ be the corresponding translation. 
	As $\widetilde A^S \rightarrow A$ is faithfully flat and $\pi_1^S(A,0)$ is pro-algebraic, there exists a preimage $a'\in \widetilde A^S(k)$, and the corresponding translation $t_{a'}$ lifts $t_a$. 
	Note that $t_{a'}$ commutes with the action of $\pi_1^S(A,0)$ on $\widetilde A^S$. 
	Consider then the automorphism $(t_{a'},\Id)$ on $\widetilde A^S\times (X_0,\Delta_0)$. 
	Clearly it commutes with the diagonal action of $\pi_1^S(A,0)$ on $\widetilde A^S\times (X_0,\Delta_0)$, hence can be descended to an automorphism $\sigma_{a'}$ on the quotient $(\widetilde A^S\times (X_0,\Delta_0))/\pi_1^S(A,0) = (X,\Delta)$. 
	Since $f$ is induced by the first projection before quotient, we must have $f\circ\sigma_{a'}=t_a \circ f$. 
	This concludes the proof.
\end{proof}

\begin{cor}\label{cor: albanese is homogeneous in text}
	Let $(X,\Delta)$ be either
	\begin{enumerate}
		\item a normal klt projective pair over $\bbC$, such that $-K_X-\Delta$ is nef, or
		\item a globally $F$-split strongly $F$-regular projective pair over an algebraically closed field $k$ of characteristic $p>0$, such that $-K_X-\Delta$ is nef and of Cartier index coprime to $p$.
	\end{enumerate}
	Then the Albanese $(X,\Delta)\rightarrow A$ is homogeneous. 
\end{cor}

\begin{proof}
	Combine Theorem \ref{thm: associated bundles are homogeneous in text} with Corollary \ref{cor: associated bundle abelian variety positive characteristic in text} and Corollary \ref{cor: associated bundle abelian variety complex numbers in text}.
\end{proof}




\printbibliography

@article{ambroModuliBDivisor2005,
  title = {The Moduli {$b$}-Divisor of an {lc}-Trivial Fibration},
  author = {Ambro, Florin},
  date = {2005},
  journaltitle = {Compositio Mathematica},
  shortjournal = {Compos. Math.},
  volume = {141},
  number = {2},
  pages = {385--403},
  langid = {english}
}

@article{beauvilleVarietesKahleriennesDont1983,
  title = {Vari{\'e}t{\'e}s K{\"a}hleriennes dont la premi{\`e}re classe de Chern est nulle},
  author = {Beauville, Arnaud},
  date = {1983},
  journaltitle = {Journal of Differential Geometry},
  shortjournal = {J. Differential Geom.},
  volume = {18},
  number = {4},
  pages = {755--782},
  langid = {french}
}

@article{biswasVectorBundlesTrivialized2011,
  title = {Vector Bundles Trivialized by Proper Morphisms and the Fundamental Group Scheme},
  author = {Biswas, Indranil and Dos Santos, Jo{\~a}o Pedro P.},
  date = {2011-04},
  journaltitle = {Journal of the Institute of Mathematics of Jussieu},
  shortjournal = {J. Inst. Math. Jussieu},
  volume = {10},
  number = {2},
  pages = {225--234},
  langid = {english}
}

@article{bogomolovDecompositionKahlerManifolds1974,
  title = {The Decomposition of K{\"a}hler Manifolds with a Trivial Canonical Class},
  author = {Bogomolov, Fedor A.},
  date = {1974},
  journaltitle = {Mathematics of the USSR-Sbornik},
  shortjournal = {Math. USSR-Sb.},
  volume = {22},
  number = {4},
  pages = {580--583},
  langid = {english}
}

@article{brionHomogeneousVectorBundles2020,
  title = {Homogeneous Vector Bundles over Abelian Varieties via Representation Theory},
  author = {Brion, Michel},
  date = {2020-02-03},
  journaltitle = {Representation Theory},
  shortjournal = {Represent. Theory},
  volume = {24},
  number = {3},
  pages = {85--114},
  langid = {english}
}

@article{caoAlbaneseMapsProjective2018,
  title = {Albanese Maps of Projective Manifolds with Nef Anticanonical Bundles},
  author = {Cao, Junyan},
  date = {2019},
  journaltitle = {Annales scientifiques de l'{\'E}cole normale sup{\'e}rieure},
  shortjournal = {Ann. Sci. {\'E}c. Norm. Sup{\'e}r.},
  volume = {52},
  number = {5},
  pages = {1137--1154},
  langid = {english}
}

@incollection{deligneTannakianCategories,
  title = {Tannakian Categories},
  author = {Deligne, Pierre and Milne, James S.},
  date = {1982},
  booktitle = {Hodge Cycles, Motives, and Shimura Varieties},
  series = {Lecture Notes in Mathematics},
  number = {900},
  publisher = {Springer-Verlag},
  pages = {101--228},
  langid = {english}
}

@article{demaillyCOMPACTCOMPLEXMANIFOLDS,
  title = {Compact Complex Manifolds with Numerically Effective Tangent Bundles},
  author = {Demailly, Jean-Pierre and Peternell, Thomas and Schneider, Michael},
  date = {1994},
  journaltitle = {Journal of Algebraic Geometry},
  shortjournal = {J. Algebraic Geom.},
  volume = {3},
  number = {2},
  pages = {295--345},
  langid = {english}
}

@article{demaillyCompactKahlerManifolds,
  title = {Compact K{\"a}hler Manifolds with Hermitian Semipositive Anticanonical Bundle},
  author = {Demailly, Jean-Pierre and Peternell, Thomas and Schneider, Michael},
  date = {1996},
  journaltitle = {Compositio Mathematica},
  shortjournal = {Compos. Math.},
  volume = {101},
  number = {2},
  pages = {217--224},
  langid = {english}
}

@online{ejiriSplittingAlgebraicFiber2023a,
  title = {Splitting of Algebraic Fiber Spaces with Nef Relative Anti-Canonical Divisor and Decomposition of {$F$}-Split Varieties},
  author = {Ejiri, Sho},
  date = {2023-08-28},
  eprint = {2308.11145},
  eprinttype = {arXiv},
  eprintclass = {math.AG},
  langid = {english}
}

@article{ejiriWhenAlbaneseMorphism2019,
  title = {When Is the {{Albanese}} Morphism an Algebraic Fiber Space in Positive Characteristic?},
  author = {Ejiri, Sho},
  date = {2019-09-01},
  journaltitle = {manuscripta mathematica},
  shortjournal = {Manuscripta Math.},
  volume = {160},
  number = {1--2},
  pages = {239--264},
  langid = {english}
}

@article{gongyoRationalConnectednessGlobally2015,
  title = {On Rational Connectedness of Globally {{{$F$}-regular}} Threefolds},
  author = {Gongyo, Yoshinori and Li, Zhiyuan and Patakfalvi, Zsolt and Schwede, Karl and Tanaka, Hiromu and Zong, Runhong},
  date = {2015-08},
  journaltitle = {Advances in Mathematics},
  shortjournal = {Adv. Math.},
  volume = {280},
  pages = {47--78},
  langid = {english}
}

@article{haraFregularFpureRings2001,
  title = {{$F$}-Regular and {{{$F$}-pure}} Rings vs. Log Terminal and Log Canonical Singularities},
  author = {Hara, Nobuo and Watanabe, Kei-ichi},
  date = {2001-12-17},
  journaltitle = {Journal of Algebraic Geometry},
  shortjournal = {J. Algebr. Geom.},
  volume = {11},
  number = {2},
  pages = {363--392},
  langid = {english}
}

@book{hartshorneAlgebraicGeometry1977,
  title = {Algebraic {{Geometry}}},
  author = {Hartshorne, Robin},
  date = {1977},
  series = {Graduate {{Texts}} in {{Mathematics}}},
  number = {52},
  publisher = {Springer New York},
  langid = {english}
}

@article{hochsterTightClosureStrong1989,
  title = {Tight Closure and Strong {{{$F$}-regularity}}},
  author = {Hochster, Melvin and Huneke, Craig},
  date = {1989},
  journaltitle = {M{\'e}moires de la Soci{\'e}t{\'e} Math{\'e}matique de France. Nouvelle S{\'e}rie},
  shortjournal = {M{\'e}m. Soc. Math. Fr. Nouv. S{\'e}r.},
  number = {38},
  pages = {119--133},
  langid = {english}
}

@book{jantzenRepresentationsAlgebraicGroups2014,
  title = {Representations of {{Algebraic Groups}}},
  author = {Jantzen, Jens Carsten},
  date = {2014},
  series = {Mathematical {{Surveys}} and {{Monographs}}},
  number = {107},
  publisher = {American Mathematical Society},
  langid = {english}
}

@article{kawamataCharacterizationAbelianVarieties1981,
  title = {Characterization of Abelian Varieties},
  author = {Kawamata, Yujiro},
  date = {1981},
  journaltitle = {Compositio Mathematica},
  shortjournal = {Compos. Math.},
  volume = {43},
  number = {2},
  pages = {253--276},
  langid = {english}
}

@book{kollarBirationalGeometryAlgebraic2002,
  title = {Birational Geometry of Algebraic Varieties},
  author = {Koll{\'a}r, J{\'a}nos and Mori, Shigefumi},
  date = {2002},
  series = {Cambridge Tracts in Mathematics},
  number = {134},
  publisher = {Cambridge University Press},
  langid = {english}
}

@article{langAlgebraicGroupsFinite1956,
  title = {Algebraic {{Groups Over Finite Fields}}},
  author = {Lang, Serge},
  date = {1956-07},
  journaltitle = {American Journal of Mathematics},
  shortjournal = {Am. J. Math.},
  volume = {78},
  number = {3},
  pages = {555--563},
  langid = {english}
}

@article{langerSfundamentalGroupScheme2011,
  title = {On the {{S-fundamental}} Group Scheme},
  author = {Langer, Adrian},
  date = {2011},
  journaltitle = {Annales de l'Institut Fourier},
  shortjournal = {Ann. Inst. Fourier},
  volume = {61},
  number = {5},
  pages = {2077--2119},
  langid = {english}
}

@article{langerSfundamentalGroupScheme2012,
  title = {On the {{S-fundamental}} Group Scheme. {{II}}},
  author = {Langer, Adrian},
  date = {2012-10},
  journaltitle = {Journal of the Institute of Mathematics of Jussieu},
  shortjournal = {J. Inst. Math. Jussieu},
  volume = {11},
  number = {4},
  pages = {835--854},
  langid = {english}
}

@article{matsumuraStructureTheoremProjective2025,
  title = {Structure Theorem for Projective klt Pairs with Nef Anti-Canonical Divisor},
  author = {Matsumura, Shin-ichi and Wang, Juanyong},
  date = {2025-09-30},
  journaltitle = {Journal of the European Mathematical Society},
  shortjournal = {J. Eur. Math. Soc.}
}

@incollection{miyanishiSomeRemarks1973,
  title = {Some Remarks on Algebraic Homogeneous Vector Bundles},
  booktitle = {Number Theory, Algebraic Geometry and Commutative Algebra: {{In}} Honor of Yasuo Akizuki},
  author = {Miyanishi, Masayoshi},
  editor = {Kusunoki, Yusuke and Mizohata, Sigeru and Nagata, Masayoshi and Toda, Hiroshi and Yamaguti, Masaya and Yoshizawa, Hiroki},
  date = {1973},
  pages = {71--93},
  publisher = {Kinokuniya Book-store Co.},
  langid = {english}
}

@article{noriFundamentalGroupscheme1982,
  title = {The Fundamental Group-Scheme},
  author = {Nori, Madhav V},
  date = {1982-07},
  journaltitle = {Proceedings Mathematical Sciences},
  shortjournal = {Proc. Math. Sci.},
  volume = {91},
  number = {2},
  pages = {73--122},
  langid = {english}
}

@article{patakfalviFsingularitiesFamilies2018,
  title = {{{$F$}}-Singularities in Families},
  author = {Patakfalvi, Zsolt and Schwede, Karl and Zhang, Wenliang},
  date = {2018-05-01},
  journaltitle = {Algebraic Geometry},
  shortjournal = {Algebr. Geom.},
  pages = {264--327},
  langid = {english}
}

@article{patakfalviSemipositivityPositiveCharacteristics2014,
  title = {Semi-Positivity in Positive Characteristics},
  author = {Patakfalvi, Zsolt},
  date = {2014},
  journaltitle = {Annales scientifiques de l'{\'E}cole normale sup{\'e}rieure},
  shortjournal = {Ann. Sci. {\'E}c. Norm. Sup{\'e}r.},
  volume = {47},
  number = {5},
  pages = {991--1025},
  langid = {english}
}

@article{patakfalviWeakBeauvilleBogomolovDecomposition2025,
  title = {The Weak {{Beauville--Bogomolov}} Decomposition in Characteristic {$p\geq 0$}},
  author = {Patakfalvi, Zsolt and Zdanowicz, Maciej},
  date = {2026},
  journaltitle = {Annales scientifiques de l'{\'E}cole normale sup{\'e}rieure},
  shortjournal = {Ann. Sci. {\'E}c. Norm. Sup{\'e}r.},
  volume = {59},
  number = {2},
  pages = {501--577},
  langid = {english}
}

@article{schwedeBehaviorTestIdeals2012,
  title = {On the Behavior of Test Ideals under Finite Morphisms},
  author = {Schwede, Karl and Tucker, Kevin},
  date = {2014},
  journaltitle = {Journal of Algebraic Geometry},
  shortjournal = {J. Algebraic Geom.},
  volume = {23},
  number = {3},
  pages = {399--443},
  langid = {english}
}

@article{schwedeCentersFpurity2010,
  title = {Centers of {{{$F$}-purity}}},
  author = {Schwede, Karl},
  date = {2010-07},
  journaltitle = {Mathematische Zeitschrift},
  shortjournal = {Math. Z.},
  volume = {265},
  number = {3},
  pages = {687--714},
  langid = {english}
}

@article{tanakaBertiniTheoremsAdmitting2024,
  title = {Bertini Theorems Admitting Base Changes},
  author = {Tanaka, Hiromu},
  date = {2024-04},
  journaltitle = {Journal of Algebra},
  shortjournal = {J. Algebra},
  volume = {644},
  pages = {64--125},
  langid = {english}
}

@article{wittenbergAlbaneseTorsorsElementary2008a,
  title = {On {{Albanese}} Torsors and the Elementary Obstruction},
  author = {Wittenberg, Olivier},
  date = {2008-04},
  journaltitle = {Mathematische Annalen},
  shortjournal = {Math. Ann.},
  volume = {340},
  number = {4},
  pages = {805--838},
  langid = {english}
}

@article{xuHomogeneousFibrationsLog2020,
  title = {Homogeneous Fibrations on Log {{Calabi--Yau}} Varieties},
  author = {Xu, Jinsong},
  date = {2020-07},
  journaltitle = {manuscripta mathematica},
  shortjournal = {Manuscripta Math.},
  volume = {162},
  number = {3--4},
  pages = {389--401},
  langid = {english}
}

@MISC{stacks-project,
    AUTHOR = "Authors, The Stacks Project",
    TITLE = "Stacks Project",
    URL = "https://stacks.math.columbia.edu/",
    SHORTHAND = "Stack"
}

\end{document}